\documentclass{article}
\usepackage[utf8]{inputenc}

\usepackage[a4paper,top=3cm,bottom=2cm,left=3cm,right=3cm,marginparwidth=1.75cm]{geometry}

\usepackage[T1]{fontenc}
\usepackage[english]{babel}
\usepackage{tikz-cd}

\usepackage{enumitem} 
\usepackage{scalerel,stackengine}
\usepackage{tikz-cd}
\usepackage{bm}
\usepackage{url}
\usepackage{amsmath}
\usepackage{amssymb}
\usepackage{amsthm}
\usepackage{mathtools}
\usepackage{amsmath,calligra,mathrsfs}
\usepackage{yhmath}

\usepackage{graphicx}
\usepackage[colorinlistoftodos]{todonotes}
\usepackage[colorlinks=true, allcolors=blue]{hyperref}

\DeclareMathOperator{\Deriv}{\mathscr{D}\text{\kern -3pt {\calligra\large er}}\,}
\DeclareMathOperator{\Endo}{\mathscr{E}\text{\kern -3pt {\calligra\large nd}}\,}
\DeclareMathOperator{\Hom}{Hom \text{\kern -3pt }\,}
\DeclareMathOperator{\Der}{Der \text{\kern -3pt }\,}
\DeclareMathOperator{\End}{End \text{\kern -3pt }\,}
\DeclareMathOperator{\Coh}{Coh \text{\kern -2pt }\,}
\DeclareMathOperator{\Gal}{Gal \text{\kern -2pt }\,}
\DeclareMathOperator{\dev}{dev \text{\kern -2pt }\,}
\DeclareMathOperator{\Kdim}{Kdim \text{\kern -2pt }\,}
\DeclareMathOperator{\GKdim}{GKdim \text{\kern -2pt }\,}
\DeclareMathOperator{\Supp}{Supp \text{\kern -2pt }\,}
\DeclareMathOperator{\Ext}{Ext \text{\kern -2pt }\,}
\DeclareMathOperator{\QCoh}{QCoh \text{\kern -2pt }\,}
\DeclareMathOperator{\WCoh}{WCoh \text{\kern -2pt }\,}
\DeclareMathOperator{\Spec}{Spec \text{\kern -2pt }\,}
\DeclareMathOperator{\Frac}{Frac \text{\kern -2pt }\,}
\DeclareMathOperator{\res}{res \text{\kern -1pt }\,}
\DeclareMathOperator{\act}{act \text{\kern -1pt }\,}
\DeclareMathOperator{\Loc}{Loc \text{\kern -2pt }\,}
\DeclareMathOperator{\Lie}{Lie \text{\kern -2pt }\,}
\DeclareMathOperator{\loc}{loc \text{\kern -2pt }\,}
\DeclareMathOperator{\ob}{ob \text{\kern -2pt }\,}
\DeclareMathOperator{\op}{op \text{\kern -2pt }\,}
\DeclareMathOperator{\ad}{ad \text{\kern -2pt }\,}
\DeclareMathOperator{\fg}{fg \text{\kern -2pt }\,}
\DeclareMathOperator{\Mod}{Mod \text{\kern -2pt }\,}
\DeclareMathOperator{\Ann}{Ann \text{\kern -2pt }\,}
\DeclareMathOperator{\Sym}{Sym \text{\kern -2pt }\,}
\DeclareMathOperator{\Dif}{Dif \text{\kern -2pt }\,}
\DeclareMathOperator{\co}{co \text{\kern -2pt }\,}
\DeclareMathOperator{\Ad}{Ad \text{\kern -2pt }\,}
\DeclareMathOperator{\gr}{gr \text{\kern -2pt }\,}
\DeclareMathOperator{\Gr}{Gr \text{\kern -2pt }\,}
\DeclareMathOperator{\im}{im \text{\kern -2pt }\,}
\DeclareMathOperator{\Image}{Image \text{\kern -2pt }\,}
\DeclareMathOperator{\coim}{coim \text{\kern -2pt }\,}
\DeclareMathOperator{\id}{id \text{\kern -2pt }\,}
\DeclareMathOperator{\bimod}{bimod \text{\kern -2pt }\,}
\DeclareMathOperator{\height}{ht \text{\kern -2pt }\,}
\DeclareMathOperator{\tensor}{tensor \text{\kern -2pt }\,}
\DeclareMathOperator{\coker}{coker \text{\kern -2pt }\,}
\stackMath
\newcommand\reallywidehat[1]{%
\savestack{\tmpbox}{\stretchto{%
  \scaleto{%
    \scalerel*[\widthof{\ensuremath{#1}}]{\kern-.6pt\bigwedge\kern-.6pt}%
    {\rule[-\textheight/2]{1ex}{\textheight}}
  }{\textheight}%
}{0.5ex}}%
\stackon[1pt]{#1}{\tmpbox}%
}
\newcommand{\uset}[1]{\underset{#1}{\otimes}{}}

\newtheorem{theorem}{Theorem}[section]
\newtheorem{lemma}[theorem]{Lemma}
\newtheorem{proposition}[theorem]{Proposition}
\newtheorem{corollary}[theorem]{Corollary}

\newtheorem{definition}[theorem]{Definition}

\newtheorem{notation}[theorem]{Notation}

\newtheorem{assumption}[theorem]{Assumption}

\title{A geometric proof of Duflo's theorem}
\author{Ioan Stanciu }
\date{October 2020}

\begin{document}

\maketitle

\begin{abstract}
Let $R$ be a commutative ring: we explain the Beilinson-Bernstein localisation mechanism for sheaves of homogeneous twisted differential operators defined over a smooth, separated, locally of finite type $R$-scheme. As an application, we give a new proof of Duflo's theorem characterising the primitive ideals of the enveloping algebra of a complex semisimple Lie algebra.
\end{abstract}


\section{Introduction}

A classical problem in ring theory is the characterisation of the primitive spectrum. In this paper, we focus on the primitive spectrum of a semisimple Lie algebra. 




Let $K$ be a field of characteristic 0 and $G$ be a connected, simply-connected, split semisimple, smooth affine algebraic group over $K$ with Lie algebra $\mathfrak{g}= \Lie(G)$. Fix $\mathfrak{g}= \mathfrak{n}^{-} \oplus \mathfrak{h} \oplus \mathfrak{n}^+$ a Cartan decomposition. In a seminal paper \cite{BGG}, the authors define the category $\mathcal{O}$ of representations for the algebra $U(\mathfrak{g})$. The building blocks are given by Verma modules $M(\lambda)=U(\mathfrak{g}) \uset{U(\mathfrak{b})}  K_{\lambda}$ for each $\lambda \in  \mathfrak{h}^*$; $\mathfrak{b}= \mathfrak{h} \oplus \mathfrak{n}^{+}$. These are highest weight modules with unique maximal submodule $N(\lambda)$ and unique simple quotient $L(\lambda)$. Moreover, this category $\mathcal{O}$ is Artinian and the set of simple objects is characterised exactly by $L(\lambda)$ for $\lambda \in \mathfrak{h}^*$. An excellent exposition of category $\mathcal{O}$ can be found in \cite{Hu1}.

The importance of category $\mathcal{O}$ in the representation theory of the ring $U(\mathfrak{g})$ can be seen in the following theorem:

\begin{theorem}[Duflo's Theorem] \cite[Theorem 4.3]{Du}
\label{ClassicalDuflo}

Let $I$ be a primitive/prime ideal of $U(\mathfrak{g})$ with $K$-rational infinitesimal central character. Then 
$$I=\Ann(L(\lambda)) \text{ for some } \lambda \in \mathfrak{h}^*. $$

\end{theorem}

Duflo's original statement requires the ground field to be $\mathbb{C}$. In their paper \cite{BG}, Bernstein and Gelfand extend this result for algebraically closed fields of characteristic $0$ which in turn can be extended to the generality stated by a base change argument. Another purely algebraic proof of Duflo's theorem can be found in \cite{Jos} and for a categorical proof, see \cite{Gin}. One should note that if $K$ is algebraically closed, all primitive ideals have $K$-rational central character, so the theorem gives a full classification of the primitive spectrum.

\textbf{Statement of the main results}

The proofs in the literature use deep knowledge about the structure of category $\mathcal{O}$ and Harish-Chandra bimodules. The Beilinson-Bernstein localisation theorem allows for the study of representation theory of $U(\mathfrak{g})$ using the geometry of $\mathcal{D}$-modules. Along with the BB localisation, the heart of the argument lies in the following informal description: "$G$-equivariant $\mathcal{D}$-modules on the double flag variety correspond to $B$-equivariant $\mathcal{D}$-modules on the flag variety".

For $\lambda$ dominant, we let $\chi_{\lambda}:Z(U(\mathfrak{g})) \to K$ denote the corresponding central character and $U(\mathfrak{g})^{\lambda}:=U(\mathfrak{g})/ \ker_{\chi_\lambda} U(\mathfrak{g})$.

\begin{proposition}[Lemma \ref{d2computeFonideals} - Corollary \ref{d2annihilatorofidealverma} ]
\label{injectionbetweenlatticeoftwosidedidealsandsubmodulesofverma}

The function $$\mathcal{F}: \{\text{two-sided ideals in } U(\mathfrak{g})^{\lambda} \} \to \{\text{submodules of } M(\lambda) \}, \qquad \mathcal{F}(I)=I M(\lambda) $$

is injective.
\end{proposition}

This was Dixmier's initial idea of proving Duflo's theorem. We give a new proof of this proposition using geometric tools. This approach has been sketched in \cite{BoBr} in the case $\lambda=0$. Our approach uses the language of $\lambda$-twisted differential operators on the flag variety and works for all $\lambda$ dominant. As a corollary, we obtain Duflo's theorem:

\begin{corollary}[Corollary \ref{d2twistedduflotheorem}]
Let $\lambda:\mathfrak{h} \to K$ be a dominant weight and $I$ a prime ideal in $U(\mathfrak{g})^{\lambda}$. Then:

$$I=\Ann(L(\mu)) \text{ for some } \mu:\mathfrak{h} \to K.$$
\end{corollary}

\textbf{Strategy of the proof}

A secondary goal of the paper is to construct a good framework that we will use to prove an affinoid version of Duflo's theorem in \cite{Sta3}. Specifically, we explain the well-known Beilinson-Bernstein localisation mechanism in a more general setting.  We will work over a commutative ring $R$ of arbitrary characteristic rather than over an algebraically closed field of characteristic $0$. We will also consider not just classical twisted differential operators, but we will deform both the enveloping algebra and the sheaf of twisted differential operators. Let us summarise the main technical results used for proving Proposition \ref{injectionbetweenlatticeoftwosidedidealsandsubmodulesofverma}.

Let $R$ be a commutative ring and let $G$ be a connected, simply-connected, split semisimple, smooth affine algebraic group defined over $R$ and let $r \in R$ be a regular element. Fix $B$ a Borel subgroup of $G$ and let $X=G/B$ be the flag scheme.

Let $\mathcal{D}$ be a sheaf of \emph{$r$-deformed $G$-homogeneous twisted differential operators ($G$-htdo)} on the double flag variety $X \times X$, see \cite[Section 6]{Sta1} for the definition of an $r$-deformed $G$-htdo.  Let $i_r:X \to X \times X$ be the inclusion of $X$ into the right copy. One can define the pullback of $\mathcal{D}$, $i_r^{\#} \mathcal{D}$ as a sheaf of rings on $X$, see \cite[Section 7]{Sta1} for the definition of the pullback.

\begin{theorem}[Theorem \ref{d2alaBorho-Brylinski}]
\label{introtheorem1}

The pullback $i_r^{\#} \mathcal{D}$ is an $r$-deformed $B$-htdo on $X$ and there is an equivalence of categories between $G$-equivariant coherent $\mathcal{D}$-modules and $B$-equivariant coherent $i_r^{\#} \mathcal{D}$-modules.
\end{theorem}

In our applications, $\mathcal{D}$ will be a sheaf of $r$-deformed $\lambda$-\emph{twisted} differential operators on $X \times X$ for some $\lambda \in (r\mathfrak{h} \times r\mathfrak{h})^*$. Here $\mathfrak{h}=\Lie(H)$ is the Cartan subalgebra correspoding to a Cartan subgroup of $G$. In particular, when $r=1$, $R=\mathbb{C}$  and $\mathcal{D}=\mathcal{D}_{X \times X}$ we obtain the Borho-Brylinski equivalence \cite[Proposition 3.6]{BoBr}.

Next, we consider the $r$-deformation of the enveloping algebra of $\mathfrak{g}:=\Lie(G)$. This is isomorphic with $U(r\mathfrak{g})$; further we let $L$ be a closed subgroup of $G$ and $\mathcal{A}$ be a sheaf of $r$-deformed $L$-htdo on $X$. We prove in Proposition \ref{d2classicalloccomring} and \ref{d2classicalglocomring} that $L$-equivariance is preserved under localisation and taking global sections.

Therefore, we may consider the following diagram:

\begin{center}
\begin{tikzcd}
  &\Mod_{\fg}(U(\mathfrak{g} \times \mathfrak{g}),G)  \arrow[d,"\Loc"]  &\Mod_{\fg}(U(\mathfrak{g}),B).\\
  &\Coh(\mathcal{D},G) \arrow[r,"i_r^{\#}"] &\Coh(i_r^{\#}\mathcal{D},B) \arrow[u,"\Gamma"].
\end{tikzcd}    
\end{center}

By setting $R=K$ to be a field of characteristic $0$ and specialising at some $\lambda \in (\mathfrak{h} \times \mathfrak{h})^*$ dominant, we obtain Proposition \ref{injectionbetweenlatticeoftwosidedidealsandsubmodulesofverma} by considering the composition $\Gamma \circ i_r^{\#} \circ \Loc$ and noting that any two-sided ideal can be viewed as a $G$-equivariant $U(\mathfrak{g} \times \mathfrak{g})$-module. We should also remark that when $\lambda$ is also regular, all the arrows in the diagram are equivalences of categories.







\textbf{Structure of the paper}

The paper is organised as follows: in Section \ref{d2sectionbackground}, we introduce deformations and review the main results in \cite{Sta1}. Next, we prove in Section \ref{d2sectionBoBr} a more general version of Borho-Brylinski equivalence \cite[Proposition 3.6]{BoBr}. In Section \ref{d2sectionidealsinevelopingalgebra}, we introduce the notion of $G$-equivariant modules over deformed enveloping algebras and prove that certain two-sided ideals are $G$-equivariant as bimodules over deformed enveloping algebras. In Section \ref{d2sectionlocmec}, we explain the localisation mechanism connecting equivariant modules over deformed enveloping algebras with equivariant modules over deformed homogeneous twisted differential operators. Next, in Section \ref{d2sectionpullbackfromdoubleflagvariety} we combine the results in Sections \ref{d2sectionBoBr} and \ref{d2sectionlocmec} to compute the pullback of certain equivariant modules over deformed homogeneous twisted differential operators on the double flag variety under the left/right inclusion of the flag variety. Finally, in Section \ref{d2sectionproofofDuflos}, we complete the proof of classical Duflo's theorem.

\textbf{Conventions.} Throughout this document $R$ will denote a commutative Noetherian base ring of arbitrary characteristic. All the varieties/algebraic groups will be considered over $R$ unless otherwise specified. Unadorned tensor products and scheme products will be assumed
to be taken over $R$ and over Spec$(R)$, respectively. 

All the algebraic groups appearing in this document will be assumed to be connected and locally of finite type unless otherwise stated. All the modules will be regarded as left modules unless explicitly stated otherwise.

For a ring $R$, we will use $R^{\op}$ to denote the opposite ring. Given an $R$-algebra $A$ and any $R$-Lie algebra/module $\mathfrak{g}/M$, we define $\mathfrak{g}_A=\mathfrak{g} \uset{R} A$ and $M_A=M \uset{R} A$. Further, we will assume that all the ring filtrations appearing in this document are positive, exhaustive and separated.

For a scheme $X$ and $\mathcal{D}$ a sheaf of rings on $X$, a $\mathcal{D}$-module is called quasi-coherent if it is quasi-coherent as an $\mathcal{O}_X$-module.

Lastly, given $f:X \to Y$ a map of schemes, we use $f^*$ to denote the pullback in the category of $\mathcal{O}/\mathcal{D}$-modules and $f^{-1}/f_*$ to denote the inverse/direct image sheaf.

\section{Modules over deformed homogeneous twisted differential operators}
\label{d2sectionbackground}
\subsection{Deformations}
\label{subsectiondeformations}

\begin{definition}
Let $A$ be a positively $\mathbb{Z}$-filtered $R$-algebra with $F_0 A$ an $R$-subalgebra of $A$. We call $A$ a \emph{deformable} $R$-algebra if $\gr A$ is a flat $R$-module. A morphism of deformable $R$-algebras is an $R$-linear filtered ring homomorphism.
\end{definition}

\begin{definition}
Let $A$ be a deformable $R$-algebra and let $r \in R$ be a regular element. The $r$-th deformation of $A$ is the following $R$-submodule of $A$:

 $$A_r:= \sum_{i=0}^{\infty} r^i F_i A.$$

\end{definition}

By construction $A_r$ is a $R$-subalgebra of $R$. Further, the definition is clearly functorial, and the following lemma states that we have a family of endofunctors $A \mapsto A_r$.

\begin{lemma}
\label{associatedgradedofdeformations}

Let $A$ be a deformable $R$-algebra and $r \in R$ a regular element. Then $A_r$ is also a deformable $R$-algebra and there is a natural isomorphism $\gr A \cong \gr A_r$.

\end{lemma}

\begin{proof}
We give $A_r$ the subspace filtration $F_i A_r:= F_i A \cap A_r$. As $\gr A$ is flat over $R$ we have $F_i A_r= \sum_{j=0}^i r^j F_j A$. For $i \geq 1$ define a $R$-linear map 
$$f:F_i A/ F_{i-1} A \to F_i A_r/F_{i-1} A_r, \quad f(x+F_{i-1} A)=r^i x +F_{i-1} A_r.$$

To finish the proof it is enough to check that $f$ is bijective. First, we prove that $f$ is injective. Assume that $r^i x \in F_{i-1} A_r $, so $r^i x \in F_{i-1} A$ which implies that $x \in F_{i-1} A$ since $\gr A$ is flat, so in particular $R$-torsion free. It is straightforward to see that $f$ is also surjective. 
\end{proof}

We recall the main constructions and results from \cite{Sta1}. We call an $R$-scheme $X$ that is smooth, separated and locally of finite type an \emph{$R$-variety}. For the rest of the section, we let $r \in R$ be a regular element.

\subsection{Deformed twisted differential operators}

Throughout this subsection, fix $X$ an $R$-variety. We write $\mathcal{T}_X$ for the sheaf of sections of the tangent bundle $TX$.

\begin{definition}\cite[Definition 4.2]{Annals}
\label{crystallinedifferentialoperatorsdefinition}

Let $X$ be an $R$-variety. The sheaf of crystalline differential operators is defined to be the enveloping algebra $\mathcal{D}_X$ of the Lie algebroid $\mathcal{T}_X$.
\end{definition}

We can view $\mathcal{D}_X$  as a sheaf of ring generated by $\mathcal{O}_X$ and $\mathcal{T}_X$ modulo the relations:

\begin{itemize}
\item{$f \partial =f \cdot \partial$;}
\item{$ \partial f - f \partial= \partial (f)$;}
\item{$ \partial \partial' - \partial' \partial=[\partial,\partial'],$}
\end{itemize}

for all $f \in \mathcal{O}_X$ and $\partial,\partial' \in \mathcal{T}_X$. The sheaf $\mathcal{D}_X$ comes equipped with a natural PBW filtration:

 $$0 \subset F_0 (\mathcal{D}_X) \subset  F_1 (\mathcal{D}_X) \subset \ldots $$

consisting of coherent $\mathcal{O}_X$-modules such that

$$F_0(\mathcal{D}_X)= \mathcal{O}_X, \quad F_1( \mathcal{D}_X)= \mathcal{O}_X \oplus \mathcal{T}_X, \quad F_m(\mathcal{D}_X) = F_1 (\mathcal{D}_X) \cdot F_{m-1}( \mathcal{D}_X)  \text{ for } m >1.                   $$

Since $X$ is smooth, the tangent sheaf $\mathcal{T}_X$ is locally free and the associated graded sheaf of algebras of $\mathcal{D}_X$ is isomorphic to the symmetric algebra of $\mathcal{T}_X$:

\begin{equation}
\label{gradingcrysdifop}
 \gr(\mathcal{D}_X)= \bigoplus_{m=1}^{\infty} \frac{F_m(\mathcal{D}_X)}{F_{m-1}(\mathcal{D}_X)} \cong \Sym_{\mathcal{O}_X} \mathcal{T}_X.
\end{equation}

If $q:T^*X \to X$ is the cotangent bundle of $X$ defined by the locally free sheaf $\mathcal{T}_X$, then we can also identify $\gr(\mathcal{D}_X)$ with $q_{*}\mathcal{O}_{T^*X}$.

Let $X$ be an $R$-variety and let $U=\Spec(A) \subset X$ be open affine. Further, we consider $\mathcal{M}$ a sheaf of $\mathcal{O}_X$-bimodules quasi-coherent with respect to the left action. We define a filtration on $M=\mathcal{M}(U)$ given by $F_{\bullet}M$:
\begin{itemize}
\item{$F_{-1}(M)=0,$}
\item{$F_n(M)= \{ m \in M | \ad(a_0)\ad(a_1)...\ad(a_n)(m)=0 \text{ for any } a_0,a_1, \ldots a_n \in A\}$}, for $n \geq 0$.

\end{itemize} 

We say that $M$ is differential if $M= \cup_{n \in \mathbb{N}^*} F_n(M)$ and we call $\mathcal{M}$ a differential $\mathcal{O}_X$-bimodule if there is an affine open cover $(U_i)_{i \in I}$ such that $\mathcal{M}(U_i)$ is a differential bimodule for all $i \in I$.

Let $\mathcal{M},\mathcal{N}$ be quasi-coherent $\mathcal{O}_X$-modules. For any affine open $U$, the set $\Hom_{R}(\mathcal{M}(U),\mathcal{N}(U))$ has the structure of a $\mathcal{O}_X(U)$-bimodule. Let $\mathcal{F} \in \Hom_{R}(\mathcal{M}, \mathcal{N}$); we say that $\mathcal{F}$ is a differential operator of degree $\leq n$ if for any affine open $U$, $\mathcal{F}(U) \in F_n(\Hom_{R}(\mathcal{M}(U),\mathcal{N}(U)).$


\begin{definition}
Let $\mathcal{A}$ be a $\mathcal{O}_X$-algebra. We say that $\mathcal{A}$ is a differential algebra if $\mathcal{A}$ is a flat $R$-module and multiplication makes $\mathcal{A}$ a differential $\mathcal{O}_X$-bimodule. The filtration $F_{\bullet}(A)$ becomes a ring filtration and with respect to this filtration $\gr^{F}(A)$ is commutative.

\end{definition}


\begin{definition}
\label{d2tdodefinition}
An algebra of $r$-deformed twisted differential operators ($r$-deformed tdo) is an $\mathcal{O}_X$-differential algebra $\mathcal{D}$ such that:
\begin{enumerate}[label=\roman*)]
\item{ The natural map $\mathcal{O}_X \to F_0(\mathcal{D})$ is an isomorphism.}
\item{ The morphism $\gr_{1}^F \mathcal{D} \to \mathcal{T}_X=\Der_R(X,X)$ defined by $\psi \mapsto \ad_{\psi}$ for $\psi \in F_1(\mathcal{D})$ induces an isomorphism $\gr_{1}^F \mathcal{D} \to r \mathcal{T}_X.$}
\item{The morphism of $\mathcal{O}_X$-algebras $\Sym_{\mathcal{O}_X}(\gr_1^F \mathcal{D}) \to \gr^F \mathcal{D}$ is an isomorphism.}
\end{enumerate}

\end{definition}

Further in the document we will use the following lemma:

\begin{lemma}
\label{d2deformationoftdoisrtdo}
Let $\mathcal{D}$ be a tdo($1$-tdo). Let $\mathcal{D}_r$ be the sheafification  obtained by composing $\mathcal{D}$ with the deformation functor in subsection \ref{subsectiondeformations}. Then $\mathcal{D}_r$ is an $r$-deformed tdo.
\end{lemma}

\begin{proof}
The claim is local, so we may assume that $X$ is affine and $D=\mathcal{D}(X)$ is the sheaf of crystalline differential operators on $X$. Clearly, $D_r$ is a differential algebra. Next, we have $F_1 D= \mathcal{O}(X) \oplus \mathcal{T}(X)$. By construction we have $F_0 D_r= \mathcal{O}(X)$ and $F_1 D_r= \mathcal{O}(X) \oplus r \mathcal{T}(X)$, so the first two axioms of Definition \ref{d2tdodefinition} are satisfied. Since $r$ is regular, $r \gr_1 D_r \cong \gr_1 D$ as $\mathcal{O}(X)$-modules, so the last claim follows from Lemma \ref{associatedgradedofdeformations}.
\end{proof}

\subsection{Equivariant \texorpdfstring{$\mathcal{O}$}{O}-modules}
Let $G$ be an affine algebraic group scheme acting on a  scheme $X$; denote the action by $\sigma_X: G \times X \to X$. Furthermore, we denote $p_X:G \times X \to X$ and $p_{2X}:G \times G \times X \to X$ the projections on the $X$ factor, $p_{23X}:G \times G \times X \to G \times X$ the projection onto the second and third factor and $m:G \times G \to G$ the multiplication of the group $G$.

\begin{definition}

Let $G$ an algebraic group scheme acting on a scheme $X$. A $G$-equivariant $\mathcal{O}_X$-module is a pair $(\mathcal{M},\alpha)$ where $\mathcal{M}$ is a quasi-coherent $\mathcal{O}_X$-module and $\alpha:\sigma_X^*\mathcal{M} \to p_X^*\mathcal{M}$ is an isomorphism of $\mathcal{O}_{G \times X}$-modules such that the diagram

\begin{center}
\begin{tikzcd}

&(1_G \times \sigma_X)^*p_X^*\mathcal{M} \arrow[r," p_{23X}^* \alpha"]   &p_{2X}^*\mathcal{M} \\
&(1_G \times \sigma_X)^* \sigma_X^* \mathcal{M} \arrow [u,"(1_G \times \sigma_X)^* \alpha "] \arrow[r,leftrightarrow,"id "] &(m \times 1_X)^*\sigma_X^* \mathcal{M} \arrow[u, " (m \times 1_X)^* \alpha"]
\end{tikzcd}
\end{center}

of $\mathcal{O}_{G \times G \times X}$-modules commutes (the cocycle condition) and the pullback 
                        $$(e \times 1_X)^* \alpha: \mathcal{M} \to \mathcal{M}$$
is the identity map.
\end{definition}

\begin{lemma}\cite[Lemma 2.2]{Sta1}
\label{Oequivpreservefunctor}

Let $G$ be an affine algebraic group acting on schemes $X$ and $Y$ and let $f:Y \to X$ be a $G$-equivariant morphism. Then the pullback functor $f^*$ given by $$(\mathcal{M},\alpha) \mapsto (f^*\mathcal{M},(1_G \times f)^* \alpha)$$ defines a functor from $G$-equivariant $\mathcal{O}_X$-modules to $G$-equivariant $\mathcal{O}_Y$-modules.

\end{lemma}

\begin{definition}
Let $G$ an affine algebraic group acting on a scheme $X$ via $\sigma_X$. We define the category of $G$-equivariant quasi-coherent $\mathcal{O}_X$-modules. Objects are given by $G$-equivariant $\mathcal{O}_X$-modules.

A morphism of $G$-equivariant $\mathcal{O}_X$ modules $(\mathcal{M},\alpha_M)$ and $(\mathcal{N},\alpha_N)$ is a map $\phi \in \Hom_{\mathcal{O}_X}(\mathcal{M},\mathcal{N})$ such that the following diagram commutes:
\begin{center}
\begin{tikzcd}
&\sigma_X^* \mathcal{M} \arrow[d,"\sigma_X^*\phi"] \arrow[r,"\alpha_M"] &p_X^* \mathcal{M} \arrow[d,"p_X^* \phi"] \\
&\sigma_X^* \mathcal{N}  \arrow[r,"\alpha_N"] &p_X^* \mathcal{N}. 
\end{tikzcd} 
\end{center}

We call such a morphism $G$-equivariant and denote the category of $G$-equivariant $\mathcal{O}_X$-modules together with $G$-equivariant morphisms by $\QCoh(\mathcal{O}_X,G)$.


\end{definition}

\begin{proposition}\cite[Proposition 2.4]{Sta1}
\label{GequivariantAbeliancat}

Let $G$ an affine algebraic group scheme acting on a scheme $X$. Then the category $\QCoh(\mathcal{O}_X,G)$ is Abelian.

\end{proposition}

From now on, when we use the notion of morphism of $G$-equivariant $\mathcal{O}_X$-modules, we always view it as a morphism in the category $\QCoh(\mathcal{O}_X,G)$.

\textbf{A reformulation of equivariance}

We wish to reformulate the notion of an equivariant $\mathcal{O}$-module. Until the end of the section, we fix $X$ a scheme defined over $R$ acted on by an affine algebraic group $G$. We start with a very simple observation: viewing $\mathcal{O}_X$ as a left $\mathcal{O}_X$-module, $(\mathcal{O}_X,\id)$ is a $G$-equivariant $\mathcal{O}_X$-module. We can reformulate this following ideas in \cite{MvdB}: for each $R$-algebra $A$ inducing a map $s:\Spec A \to \Spec R$ and for each point $i_{g}:\Spec A \to G$ which induces an automorphism $g:X_{A} \to X_{A}$ there exists an isomorphism

                                                                    $$q_g: s^* \mathcal{O} \to (g^{-1})^*s^* \mathcal{O}, \text{ satisfying }$$

\begin{equation}
\label{structureequiveq}
              q_e=\id \text{ and } q_{gh}=(g^{-1})^*(q_h)q_g
\end{equation}              
in such a way that $(q_g)$'s are compatible with base change. Let $r_g=g^* \circ q_g$. For each $U \subset X_A$ affine open, $r_g$ is a map $\mathcal{O}_A(U) \to \mathcal{O}_A(g^{-1}U)$. The equation \ref{structureequiveq} translates as $r_e=\id$ and $r_{gh}=r_hr_g$. Furthermore, the $\mathcal{O}$-module compatibility requires that for any $f_1,f_2 \in \mathcal{O}_A(U)$, we have $r_g(f_1f_2)=r_g(f_1)r_g(f_2)$.  

We define $r_g$ via $r_g(f)(x)=f(g^{-1}x)$ for all $R$-algebras $A$, $U \subset X_A$ affine open, $x \in U$, $f \in \mathcal{O}_A(U)$, $g:X_{A} \to X_{A}$ and it is easy to see that $r_g$'s make $\mathcal{O}_X$ a $G$-equivariant $\mathcal{O}_X$-module. We may now make an abuse of notation: for each $i_g: \Spec A \to G$ and each $ f \in \mathcal{O}_A(U)$, we denote $g.f=r_{g^{-1}}(f)$ and we translate  the equivariance structure as

$$ e.f_1=f, \quad g.(h.f_1)=(gh).f_1, \quad g.(f_1f_2)=(g.f_1)(g.f_2) \text{ for all  } g,h \in G,  f_1,f_2 \in \mathcal{O}_X.$$

\begin{lemma}\cite[Lemma 2.5]{Sta1}
A $\mathcal{O}_X$-module  $\mathcal{M}$ is $G$-equivariant if and only if for each $R$-algebra $A$, for each $s:\Spec A \to \Spec R$ and for each point $i_g: \Spec A \to G$ which induces an automorphism $g:X_A \to X_A$ there exists an isomorphism of $\mathcal{O}_{A}$-modules

      $$ q_g:s^* \mathcal{M} \to (g^{-1})^* s^* \mathcal{M}        $$

satisfying

\begin{equation}
\label{weaklyequiveqOmod}
q_{e}=\id \text{ and }  q_{gh}=(g^{-1})^*(q_h)q_g
\end{equation}

in such a way that $(q_g)$'s are compatible with base change.

\end{lemma}

Again by setting $s_g=g^* \circ q_g$, we may reformulate equation \ref{weaklyequiveqOmod} as:  for each $R$-algebra $A$ and for each $i_g:\Spec A \to G$, we have an isomorphism of $\mathcal{O}$-modules $s_g: \mathcal{M}_{X_A} \to \mathcal{M}_{X_A}$ such that for each $U \subset X_A$ affine open:

\begin{equation}
\label{alternativeeqOmod}
\begin{split}
&s_{e}=\id,\\
&s_{gh}=s_{h}s_{g},\\
&s_{g}'s \text{ are compatible with base change},\\
& r_g(f.m)=r_g(f).s_g(m) \text{ for all }  f \in \mathcal{O}_{Y_A}(U), m \in \mathcal{M}(U).
\end{split}
\end{equation}

Again, we make an abuse of notation:  for each $i_g: \Spec A \to G$ and each $ m \in \mathcal{M}_A(U)$, we denote $g.m=s_{g^{-1}}(m)$ and we translate  the equivariance structure as:

\begin{equation}
\label{easyweakOmodequivariant}
\begin{split}
&e.m=m, \\
&gh.m= g.(h.m),\\
&g.(f.m)=(g.f).(g.m),
\end{split}
\end{equation}

for all $g,h \in G$, $m \in \mathcal{M}$,  $ f \in \mathcal{O}_X$.

\subsection{Deformed Homogeneous twisted differential operators}

Throughout this subsection, we will assume that $G$ is a connected, smooth affine algebraic group scheme with Lie algebra $\mathfrak{g}:=\Lie(G)$.

\begin{definition}
\label{d2htdodef}
Let $\mathcal{D}$ be a differential $\mathcal{O}_X$-algebra. We call $\mathcal{D}$ a sheaf $r$-deformed $G$-\\homogeneous twisted differential operators($G$-htdo) if it is $G$-equivariant as a left $\mathcal{O}_X$-module, $\mathcal{D}$ is an $r$-deformed tdo  and $\mathcal{D}$ is equipped with a Lie algebra map $i_{\mathfrak{g}}:r\mathfrak{g} \to \mathcal{D}$ such that:

\begin{enumerate}[label=\roman*)]
\item{$g.1=1$ and $g.(d_1d_2)=(g.d_1)(g.d_2)$ for $g \in G$ and $d_1,d_2 \in \mathcal{D}.$ }
\item{$g.(fd)=(g.f)(g.d)$ for $f \in \mathcal{O}_X$ and $d \in \mathcal{D}$.}
\item{$i_{\mathfrak{g}}(g.\psi)=g.i_{\mathfrak{g}}(\psi)$ for $g \in G, \psi \in r\mathfrak{g}.$}
\item{The derivative of the $G$-action induces a $\mathfrak{g}$-action and so a $r \mathfrak{g} \subset \mathfrak{g}$-action. This must coincide with the action $d \to [i_{\mathfrak{g}}(\psi),d]$ for $\psi \in r\mathfrak{g}$ and $ d \in \mathcal{D}$.}

\item{$i_{\mathfrak{g}}(r\mathfrak{g}) \subset F_1 \mathcal{D}.$}
\item{$\eta=\rho \circ i_{\mathfrak{g}}$ as maps from $r \mathfrak{g}$ to $r \mathcal{T}_X$ where $\eta:\mathfrak{g} \to \mathcal{T}_X$ is the infinitesimal map and $\rho: F_1 \mathcal{D} \to \mathcal{T}_X$ is the natural anchor map.}
\end{enumerate}
\end{definition}


Let $Y$ be another $R$-variety such that $G$ acts on $Y$ and $f:Y \to X$ is a $G$-equivariant morphism. Then for $\mathcal{D}$ an $r$-deformed  $G$-htdo on $X$, we defined in \cite[Definition 7.5]{Sta1} its pullback, $f^{\#} \mathcal{D}$ and proved in \cite[Corollary 7.6]{Sta1} that it is an $r$-deformed $G$-htdo on $Y$.

Assume further that $f:Y \to X$ is a locally trivial $G$-torsor (see \cite[Section 4.3]{Annals} for the definition). Then for $\mathcal{A}$ an $r$-deformed $G$-htdo on $Y$, we defined its descent, $f_{\#} \mathcal{D}^G$, \cite[Definition 10.9]{Sta1} and proved in \cite[Lemma 10.10]{Sta1} it is an $r$-deformed tdo on $X$. The main proposition we need is:

\begin{proposition}\cite[Corollary 10.13]{Sta1}
\label{eqdescentliealg}
Let $f:Y \to X$ be a locally trivial $G$-torsor. Let $B$ be another smooth affine algebraic group acting on $X$ and $Y$, such that the actions of $G$ and $B$ on $Y$ commute. The maps $f_{\#}(-)^G$ and $f^{\#}(-)$ induce inverse bijections from the set of $r$-deformed  $G \times B$-equivariant htdo's on $Y$ to the set of $r$-deformed $B$-equivariant htdo's on $X$.
\end{proposition} 

In particular, by setting $B$ to be a trivial group we obtain a bijection between the set of $r$-deformed $G$-htdo on $Y$ and the set of $r$-deformed tdo's on $X$.

\begin{definition}
\label{d2equivarianthtdomoddef}
Let $(\mathcal{D},i_{\mathfrak{g}})$ be a $r$-deformed $G$-htdo and $L$ be a closed subgroup of $G$, with Lie algebra $\mathfrak{l}$. A $\mathcal{D}$-module $\mathcal{M}$ is weakly $L$-equivariant if:

\begin{enumerate}[label=\roman*)]
\item{$\mathcal{M}$ is an $L$-equivariant  $\mathcal{O}_X$-module.}
\item{$g.(D.m)=(g.D).(g.m)$ for any $g \in L, d \in \mathcal{D}, m \in \mathcal{M}$.} 

We call $\mathcal{M}$ $L$-equivariant if in addition:

\item{ The $r\mathfrak{l}$-action induced by the derivative of the $L$-action on $\mathcal{M}$ coincides with the $r\mathfrak{l}$-action induced by the restriction of $i_{\mathfrak{g}}$ to $r\mathfrak{l}$.}

A morphism of (weakly) equivariant $\mathcal{D}$-modules is a morphism of $L$-equivariant $\mathcal{O}_X$-modules that is $\mathcal{D}$-linear.
\end{enumerate}

\end{definition}

In case $L=G$, we recover \cite[Definition 9.3]{Sta1}, but we will need this more general definition for explaining the localisation mechanism. We denote $\Coh(\mathcal{D},G)$ the category of coherent $G$-equivariant coherent $\mathcal{D}$-modules. 

Let $Y$ be another $R$-variety such that $G$ acts on $Y$, $f:Y \to X$ is a $G$-equivariant morphism and let $\mathcal{D}$ be $G$-htdo on $X$. Given a $G$-equivariant $\mathcal{D}$-module $\mathcal{M}$, we endow $f^* \mathcal{M}$ with an action $f^{\#} \mathcal{D}$ and we call this $f^{\#} \mathcal{M}$. We prove in \cite[Lemma 9.7]{Sta1} that $f^{\#} \mathcal{M}$ is $G$-equivariant.

We may redefine the notion of $G$-equivariance of an $r$-deformed $G$-htdo module. Denote the $G$-action by $\sigma_X: G \times X \to X$. Furthermore, we denote $p_X:G \times X \to X$ and $p_{2X}:G \times G \times X \to X$ the projections on the $X$ factor, $p_{23X}:G \times G \times X \to G \times X$ the projection onto the second and third factor and $m:G \times G \to G$ the multiplication of the group $G$. Then we define a $G$-equivariant $\mathcal{D}$-module as a pair $(\mathcal{M},\alpha)$, where $\mathcal{M}$ is a $\mathcal{D}$-module and $\alpha:\sigma_X^{\#} \mathcal{M} \to p_X^{\#} \mathcal{M}$ is an isomorphism of $p_X^{\#} \mathcal{D}$-modules such that the diagram:

\begin{equation}
\label{eqmoduleGhtdocommdia}
\begin{tikzcd}
&(1_G \times \sigma_X)^{\#}p_X^{\#}\mathcal{M} \arrow[r," p_{23X}^{\#} \alpha"]   &p_{2X}^{\#}\mathcal{M} \\
&(1_G \times \sigma_X)^{\#} \sigma_X^{\#} \mathcal{M} \arrow [u,"(1_G \times \sigma_X)^{\#} \alpha "] \arrow[r,leftrightarrow,"id "] &(m \times 1_X)^{\#}\sigma_X^{\#} \mathcal{M} \arrow[u, " (m \times 1_X)^{\#} \alpha"]
\end{tikzcd}
\end{equation}

commutes ( the cocycle condition) and the pullback 
                        $$(e \times 1_X)^{\#} \alpha: \mathcal{M} \to \mathcal{M}$$
is the identity map. We will ignore the equivariance structure when it is understood from the context.

The main result of \cite{Sta1} that we use in the paper is the equivariant descent for modules over deformed homogeneous twisted differential operators:

\begin{theorem}\cite[Corollary 11.9]{Sta1}
\label{eqdescentrep}
Let $G,B$ be smooth affine algebraic groups of finite type. Let $X,Y$ be $R$-varieties and let $f:Y \to X$ be a locally trivial $G$-torsor that is $B$-equivariant. Further, let $\mathcal{D}$ be a sheaf of $r$-deformed $G \times B$-homogeneous twisted differential operators on $Y$. The functors:

\begin{equation}
\begin{split}
&f_* (-)^G: \Coh(\mathcal{D},G\times B) \to \Coh(f_{\#} \mathcal{D}^G,B),\\
&f^{\#}(-): \Coh(f_{\#} \mathcal{D}^G,B) \to \Coh(\mathcal{D},G \times B).
\end{split} 
\end{equation}

are quasi-inverses equivalences of categories.

\end{theorem}

In particular setting $B=e$ the trivial group, we obtain an equivalence of categories between coherent $G$-equivariant $\mathcal{D}$-modules and coherent $f_{\#} \mathcal{D}^G$-modules. 

\section{An equivalence a la Borho-Brylinski}
\label{d2sectionBoBr}

Throughout this section, we let $G$ be a connected, smooth, affine algebraic group over $\Spec R$,  $B$ a closed subgroup of $G$ and we make the following assumption:

\begin{assumption}
\label{d2loctrivialassumption}
The quotient scheme $X=G/B$ is an $R$-variety and the quotient map $d_B:G \to X$ given by $d_B(g)=gB$ is a locally trivial $B$-torsor with respect to the action $\diamond$ given by $b \diamond g=gb^{-1}$.
\end{assumption}

This is satisfied in particular when $B$ is a Borel subgroup of $G$ and $X=G/B$ becomes the flag scheme.

\subsection{Notation and preliminaries}

First, we introduce some notation: we will denote the standard action (left multiplication) of $G$ or $B$ on $X$ by $gx$ for all $g \in G$ and $x \in X$. Furthermore, the standard $B$ action on $X$ (the restriction of the $G$-action to $B$)  will also be denoted act$_B$.

Define: 
\begin{itemize}

\item{$\act_{1,G,B}(\square):(G \times B) \times( G \times X) \to (G \times X):
(g,b) \square (h,x):=(ghb^{-1},bx) $.}
\item{$\act_{2,G,B}(\triangle):(G \times B) \times( G \times X) \to (G \times X):
(g,b) \triangle (h,x):=(ghb^{-1},gx) $.}
\item{$p:G \times X \to X, \quad p(g,x):=x.$}
\item{$\delta: G \times X \to X, \quad \delta(g,x):=g^{-1}x$.}
\item{$\alpha: G \times X \to G \times X, \quad \alpha(g,x):=(g,g^{-1}x)$.}
\end{itemize}

\begin{lemma}
\label{d2introlemma}
The following statements are easy to see:

\begin{itemize}
\item{$p \circ \alpha= \delta$.}
\item{$\alpha$ is a bijective $G \times B$-equivariant map with respect to the actions $\triangle$ and $\square$.}
\end{itemize}
\end{lemma}

\begin{lemma}
\label{d2deltaloctrivialequivarianttorsor}
The map $\delta$ is a $B$-equivariant locally trivial $G$-torsor. Further, the $\triangle$ actions of $G$ and $B$ on $G \times X$ commute.
\end{lemma}

\begin{proof}
By construction, we have that $p$ is a locally trivial $G$-torsor. The $\square$ actions of $G$ and $B$ actions on $G \times X$ commute. The claim follows by Lemma \ref{d2introlemma}, by viewing $\alpha$ as a $B$-equivariant locally trivial $G$-torsor.
\end{proof}

We introduce further notation.

\begin{itemize}
\item{Let $\act_G: G \times (X \times X) \to X \times X, \act_G(g,x,y):=(gx,gy)$  be the diagonal action .}

\item{Let $d:G \times X \to X \times X, d(h,x):=(hB,x)$ be the quotient map.}
\end{itemize}

\begin{lemma}
\label{d2dislocallytrivialBtorsor}

By abuse of notation, we will denote  $\triangle$ the $B$ action on $G \times X$ given by $b \triangle (g,x)= (gb^{-1}, x) $. Then the descent map $d$ is a $G$-equivariant locally trivial $B$-torsor with respect to the $B$-action given by $\triangle$.

\end{lemma}

\begin{proof}
We have

$$d(g \triangle(h,x))=d(ghb^{-1},gx)=(ghB,gx)=\act_G(g,d(h,x)),$$
so $d$ is indeed $G$-equivariant.

We have that $d=d_B \times \id_X$ and since $d_B$ is faithfully flat and locally of finite type so is $d$. Since $d_B$ is a $B$-torsor, we have that the map

$$ G \times B \to G \times_{X} G$$

sending $(g,b) \to (g,gb^{-1})$ is an isomorphism, so the map

$$ (G \times X) \times B \to (G \times X) \times_ {X \times X} (G \times X)$$

sending $((g,x),b) \to ((g,x),(gb^{-1},x))$ is also an isomorphism. Therefore, $d$ is indeed a $B$-torsor.

To conclude, one needs to show that $d$ is locally trivial. Let $S=\{U_i\}_{i \in I}$ be an affine open cover of $X$ trivialising the torsor $d_B$ and let $\phi_i: B \times U_i \to d_B^{-1}(U_i)$ the corresponding $B$-invariant isomorphisms. Define $\{V_{ij}:=U_i \times U_j | U_i, U_j \in S \}$. Then $\{V_{ij}\}$ is an affine open cover of $X \times X$. We have that  

$$ \mathcal{O}_{X \times X}(V_{ij})= \mathcal{O}_{X \times X} (U_i \times U_j) \cong \mathcal{O}_{X}(U_i) \uset{R} \mathcal{O}_{X}(U_j) \text{ for any } i,j \in I.$$

By assumption, for any $k \in I, \mathcal{O}_{X}(U_k)$ is a finitely generated $R$-algebra; therefore the ring $\mathcal{O}_{X}(U_i) \uset{R} \mathcal{O}_{X}(U_j)$ is also a finitely generated $R$-algebra.

We claim that any  $\{V_{ij}\}$ trivialises the torsor $d$. We have by definition $d^{-1}(V_{ij})=(d_B^{-1}U_i,U_j)$.  $$ \text{Let } \varphi_{ij}: B \times U_i \times U_j \to (d_B^{-1}U_i,U_j) \text{ be defined by } \varphi(b,x,y):= (\phi_i(b,x),y).$$  

Then it is trivial to check that this is an isomorphism compatible with the $B$-action given that $\phi_i$ is. Therefore, we have checked all the conditions to make $d$ a locally trivial $B$-torsor.   
\end{proof}

\subsection{Pullback of representations of homogeneous twisted differential operators on \texorpdfstring{$X \times X$}{X times X}}

We keep the notation from the previous subsection. Let $i_r:X \to X \times X, i_r(x)=(eB,x)$ denote the inclusion of $X$ into the right copy of $X \times X$. We also fix $(\mathcal{D},i_{\mathfrak{g}})$ an $r$-deformed $G$-htdo on $X \times X$ with respect to the $G$-action defined by $\act_G$. The goal of this subsection is to prove the following theorem:

\begin{theorem}
\label{d2alaBorho-Brylinski}
The pullback $i_r^{\#}\mathcal{D}$ is an $r$-deformed $B$-htdo and the functor $$i_r^{\#}: \Coh(\mathcal{D},G) \to \Coh(i_r^{\#}\mathcal{D},B)$$

is an equivalence of categories.

\end{theorem}

To prove this theorem, we will need some additional results:

\begin{lemma}
\label{d2eqdescentfortorsord}
The pullback $d^{\#} \mathcal{D}$ is an $r$-deformed $G \times B$-htdo and the functors
\begin{equation}
\begin{split}
d^{\#}(-): \Coh(\mathcal{D},G) \to \Coh(d^{\#} \mathcal{D},G \times B),\\
d_*(-)^B:\Coh(d^{\#} \mathcal{D},G \times B) \to \Coh(\mathcal{D},G)
\end{split}
\end{equation}

are quasi-inverses equivalences of categories.
\end{lemma}

\begin{proof}
This follows from Lemma \ref{d2dislocallytrivialBtorsor}, \cite[Corollary 10.13]{Sta1} and \cite[Corollary 11.9]{Sta1}.
\end{proof}

\begin{lemma}
\label{d2eqdescentfortorsordelta}
Denote $\mathcal{A}:=\delta_{\#} (d^{\#} \mathcal{D})^G$. Then $\mathcal{A}$ is an $r$-deformed $B$-htdo and the functors:
\begin{equation}
\begin{split}
\delta^{\#}(-): \Coh(\mathcal{A},B) \to \Coh(d^{\#} \mathcal{D},G \times B),\\
\delta_*(-)^G:\Coh(d^{\#} \mathcal{D},G \times B) \to \Coh(\mathcal{A},B)
\end{split}
\end{equation}

are quasi-inverses equivalences of categories.

\end{lemma}

\begin{proof}
This follows from Lemma \ref{d2deltaloctrivialequivarianttorsor}, \cite[Corollary 10.13]{Sta1} and \cite[Corollary 11.9]{Sta1}.
\end{proof}

As a corollary, we obtain from the previous two lemmas an equivalence between $\Coh(\mathcal{D},G)$ and $\Coh(\mathcal{A},B)$.

\begin{lemma}
\label{d2isomorphismsofhtdo}
The sheaves of $r$-deformed twisted differential operators $\mathcal{A}$ and $ i_r^{\#} \mathcal{D}$ are isomorphic. In particular, $i_r^{\#} \mathcal{D}$ is an $r$-deformed $B$-htdo.
\end{lemma}

\begin{proof}

Define the maps $s:G \times X \to G \times X \times X, s(g,x):=(g^{-1},gB,x)$ and $q:G \times X \times X \to X \times X, q(g,x,y):=(x,y)$.

By construction, $d=q \circ s$, so $d^{\#} \mathcal{D} \cong (q \circ s)^{\#} \mathcal{D}$. Since $d$ and $q$ are locally trivial $B$ respectively $G$-torsors, $s$ is flat, so by \cite[Corollary 7.7]{Sta1}, we get $(q \circ s)^{\#} \mathcal{D} \cong s^{\#} q^{\#} \mathcal{D}$. Further, since $\mathcal{D}$ is a deformed $G$-htdo, we have $q^{\#} \mathcal{D} \cong \act_G^{\#} \mathcal{D}$. Therefore, using again that $s$ is flat, we have by \cite[Corollary 7.7]{Sta1} $s^{\#}q^{\#} \mathcal{D} \cong (\act_G \circ s)^{\#} \mathcal{D}$.

By construction, one has $\act_G \circ s=i_r \circ \delta$, so $(\act_G \circ s)^{\#} \mathcal{D} \cong (i_r \circ \delta)^{\#} \mathcal{D}$. Since $\delta$ is smooth, it is in particular flat, so by \cite[Corollary 7.7]{Sta1}, $(i_r \circ \delta)^{\#} \mathcal{D} \cong \delta^{\#} i_r^{\#} \mathcal{D}$. Therefore, we have obtained that $d^{\#} \mathcal{D} \cong \delta^{\#} i_r^{\#} \mathcal{D}$. To conclude we use \cite[Corollary 10.13]{Sta1} to obtain:

\begin{equation*}
\mathcal{A}=\delta_{\#}(d^{\#} \mathcal{D})^G \cong \delta_{\#}(\delta^{\#} i_r^{\#} \mathcal{D})^G \cong i_r^{\#} \mathcal{D}. \qedhere
\end{equation*} 
\end{proof}

\begin{lemma}
\label{d2leftquasiinverseforir}
Let $\mathcal{M} \in \Coh(\mathcal{D},G)$. Then

$$(d_*^B \circ \delta^{\#}) \circ i_r^{\#} \mathcal{M} \cong \mathcal{M},$$

i.e. the functor  $d_*^B \circ  \delta^{\#}$ is a left quasi-inverse to $i_r^{\#}$.
\end{lemma}

\begin{proof}

By repeating the argument in the previous lemma, we get $d^{\#} \mathcal{M} \cong \delta^{\#} \circ i_r^{\#} \mathcal{M}$. The claim then follows by Lemma \ref{d2eqdescentfortorsord}.
\end{proof}

To complete the proof of the  equivalence theorem, we need one more categorical lemma:

\begin{lemma}
\label{d2leftinverserightinverse}
Let $C$ and $D$ be two categories and let $H:C \to D$ and $F_1,F_2: D \to C$ be functors such that:

$$HF_1 \cong 1_D, \quad F_1H \cong 1_C, \quad  HF_2 \cong 1_D.$$

Then $F_2H \cong 1_C$.
\end{lemma}

\begin{proof}

We have

$$    H( F_2 H) F_1 \cong (HF_2)(HF_1) \cong 1_D, \text{ so }               $$

\begin{equation}
\label{d2leftrightinverseeq1}
F_1 (HF_2HF_1) H \cong F_1H \cong 1_C. 
\end{equation}
On the other hand we have:

\begin{equation}
\label{d2leftrightinverseeq2}
F_1(HF_2HF_1)H \cong (F_1H)(F_2H)(F_1H) \cong 1_C(F_2H)1_C=F_2H.
\end{equation}
Thus, by combining the equations \eqref{d2leftrightinverseeq1} and \eqref{d2leftrightinverseeq2} we obtain $F_2H \cong 1_C$.
\end{proof}

We can now prove Theorem \ref{d2alaBorho-Brylinski}:

\begin{proof}
We have by Lemmas \ref{d2eqdescentfortorsord}, \ref{d2eqdescentfortorsordelta} and \ref{d2isomorphismsofhtdo} that $d_*^B \circ \delta^{\#}$ and $\delta_*^G \circ d^{\#}$ provide quasi-inverse equivalences of the categories mentioned in the statement. By Lemma \ref{d2leftquasiinverseforir}, $d_*^B \circ \delta^{\#}$ is a left quasi-inverse to $i_r^{\#}$, so by applying Lemma \ref{d2leftinverserightinverse} it is also a right quasi-inverse. Thus, for $\mathcal{M}$ a $G$-equivariant coherent $\mathcal{D}$-module, we have $i_r^{\#}\mathcal{M} \cong \delta_*^G \circ d^{\#} \mathcal{M}$. 
\end{proof}

By setting $R=\mathbb{C}$, $r=1$ and $\mathcal{D}=\mathcal{D}_{X \times X}$, we recover the classical Borho-Brylinski equivalence \cite[Proposition 3.6d)]{BoBr}.

Let $i_l:X \to X \times X$, $i_l(x):=(x,eB)$ to be the inclusion of $X$ into the left copy. Then by duality, we get:

\begin{corollary}
\label{d2leftinclusionalaBorho-Brylinski}
The pullback $i_l^{\#}\mathcal{D}$ is a $B$-htdo and the functor $$i_l^{\#}: \Coh(\mathcal{D},G) \to \Coh(i_l^{\#}\mathcal{D},B)$$

is an equivalence of categories.
\end{corollary}

\subsection{Global sections of \texorpdfstring{$\mathcal{O}$-modules}{O}}

We keep the notation from the previous subsection. We aim to prove the following proposition:

\begin{proposition}
\label{d2zeroglobalsections}
Let $\mathcal{N} \in \QCoh(\mathcal{O}_X)$ such that $\Gamma(X, \mathcal{N})=0$. Then
            $$\Gamma(X \times X,(d_*)^B \delta^{*} \mathcal{N})=0.$$

\end{proposition}

We will need two additional lemmas:

\begin{lemma}
\label{d2globalsectionsprojections}
Let $\mathcal{N}$ be a quasi-coherent $\mathcal{O}_X$-module such that $\Gamma(X, \mathcal{N})=0$. Then 

 $$\Gamma(G \times X, p^* \mathcal{N})=0.$$

\end{lemma}

\begin{proof}

Let $(U_i)_{i \in I}$ be an affine finite open cover of $X$. Then $(G \times U_i)_{i \in I}$ is an affine finite open cover of $G \times X$ stable under the map $p$. Let $U_{ij}:=U_i \cap U_j$.

Consider the Cech complex

$$0 \to \mathcal{N}(X) \to \prod_{i \in I} \mathcal{N}(U_i) \to \prod_{i,j \in I} \mathcal{N}(U_{ij}).$$ 

The assumption tells us that the last map is injective.

Consider now the Cech Complex
\begin{equation}
\label{d2Cechcomplexequation}
0 \to p^*\mathcal{N}(G \times X) \to \prod_{i \in I} p^*(\mathcal{N})(G \times U_i) \to \prod_{i,j \in I} p^*(\mathcal{N})(G \times U_{ij}).
\end{equation}

We have $p^*(\mathcal{N})(G \times U_i)= \mathcal{O}(G) \uset{R} \mathcal{N}(U_i)$ and since tensor product commutes finite products we get

$$ \prod_{i \in I} p^*(\mathcal{N})(G \times U_i) \cong  \mathcal{O}(G) \uset{R} \prod_{i \in I} \mathcal{N}(U_i),$$

and similarly

$$ \prod_{i,j \in I} p^*(\mathcal{N})(G \times U_{ij}) \cong  \mathcal{O}(G) \uset{R} \prod_{i,j \in I} \mathcal{N}(U_{ij}).$$

Since $\mathcal{O}(G)$ is flat over $R$ we get that the tensor product preserves injections, so  the last map in equation \eqref{d2Cechcomplexequation} is injective;  the claim follows.
\end{proof}

\begin{lemma}
\label{d2globalsectionsautomorphism}
Let $Y$ be a smooth scheme let $\beta:Y \to Y$ be an automorphism and let $\mathcal{M}$ be a quasi-coherent $\mathcal{O}_Y$-module such that $\Gamma(Y, \mathcal{M})=0$. Then $ \Gamma(Y,\beta^* \mathcal{M})=0.$

\end{lemma}

\begin{proof}
Since $\beta$ is an automorphism of a smooth scheme, it is smooth, so in particular it is faithfully flat. The same is true for $\beta^{-1}$. Therefore, we get an injection $$\alpha: \Gamma(Y, \beta^* \mathcal{M}) \to \Gamma(Y, \beta^{-1^*} \beta^* \mathcal{M}) \cong \Gamma(Y, \mathcal{M})=0,$$

so the claim is proven.
\end{proof}

\begin{proof}[Proof of Proposition \ref{d2zeroglobalsections}]
 Since $\delta^* \mathcal{N}=\alpha^*p^* \mathcal{N}$, we have by Lemma \ref{d2globalsectionsprojections} and Lemma \ref{d2globalsectionsautomorphism} that $\Gamma(G \times X, \delta^* \mathcal{N})=0$. The claim follows from the definition of pushforward sheaf.
\end{proof}

As a corollary we obtain:

\begin{corollary}
Let $\mathcal{N} \in \Coh(\mathcal{D},G)$ with $\Gamma(X,i_r^{\#} \mathcal{N})=0$. Then $\Gamma(X \times X, \mathcal{N})=0$.

\end{corollary}

\begin{proof}
We have that at the level of $\mathcal{O}$-modules that $i_r^{\#}=i_r^*$ and $\delta^{\#}=\delta^{*}$. The claim then follows from Theorem \ref{d2alaBorho-Brylinski} and Proposition \ref{d2zeroglobalsections}.
\end{proof}

By duality we obtain:

\begin{corollary}
\label{d2zeroglobalsectionsil}

Let $\mathcal{N} \in \Coh(\mathcal{D},G)$ with $\Gamma(X,i_l^{\#} \mathcal{N})=0$. Then $\Gamma(X \times X, \mathcal{N})=0$.

\end{corollary}


\section{Representations of deformed enveloping algebras}
\label{d2sectionidealsinevelopingalgebra}
\subsection{Notation}

Let $G$ be a connected, simply connected, split semisimple, smooth affine algebraic group over a commutative ring $R$ and $\mathfrak{g}$ denote its Lie algebra which is free as an $R$-module. Fix $\mathfrak{h}$ a Cartan subalgebra and $\mathfrak{g}= \mathfrak{n}^- \oplus \mathfrak{h} \oplus \mathfrak{n}$ a triangular decomposition. Let $\mathfrak{b}=\mathfrak{h} \oplus \mathfrak{n}$ be a Borel subalgebra.  We make the following notation:

\begin{notation}
\begin{itemize}

\item $\mathfrak{h}^*$ the dual space of $\mathfrak{h}$- elements of $\mathfrak{h}^*$ are called weights.
\item $\bm{\phi}$ - set of roots in $\mathfrak{g}$.
\item $\bm{\phi}^+$- set of positive roots.
\item $\triangle$- set of simple roots.
\item For a root $\alpha$, we denote $\alpha^{\vee}$ the corresponding coroot.
\item $\rho$- half sum of positive roots.
\item $W$- the Weyl group associated with $\mathfrak{g}$.
\item $w_{o}$- the longest element of $W$.
\end{itemize}
\end{notation}

\begin{definition}
We define a shifted dot action of $W$ on $\mathfrak{h}^*$ by 
               $$ w \cdot \lambda = w(\lambda+\rho) - \rho \text{ for } w \in W, \lambda \in \mathfrak{h}^*.$$
               
We say that a weight is \emph{dominant} if $(\lambda+\rho)(\alpha^{\vee}) \notin \mathbb{Z}^{\leq -1}$ for all $\alpha \in \bm{\phi}^+$. Any $W$-orbit contains a dominant weight.

We say that a weight is \emph{regular} if $(\lambda+\rho) (\alpha^{\vee}) \neq 0$ for all $\alpha \in \bm{\phi}^+$. This is equivalent to the stabiliser of $\lambda$ under the shifted dot action being trivial.             
\end{definition}

Let $Z(\mathfrak{g})=Z(U(\mathfrak{g}))$ denote the center of the universal enveloping algebra. For any $\lambda \in \mathfrak{h}^*$ there is an associated central character $\chi_{\lambda}:Z(\mathfrak{g}) \to R$.  Furthermore, we have that for all $w \in W$, $\chi_{\lambda}=\chi_{w \cdot \lambda}$.

\subsection{Equivariant modules over deformed enveloping algebras}
Fix $r \in R$ a regular element and consider the $r$-th deformation of $U(\mathfrak{g})$ denoted $U(\mathfrak{g})_r$. Using the PBW theorem we obtain that $U(\mathfrak{g})_r \cong U(r \mathfrak{g})$.
The enveloping algebra $U(\mathfrak{g})$ is a $G$-representation via the Adjoint action, so by the module-comodule duality we obtain a  map $\rho: U(\mathfrak{g}) \to \mathcal{O}(G) \uset{R} U(\mathfrak{g})$ making $U(\mathfrak{g})$ a comodule for the Hopf algebra $\mathcal{O}(G)$. Furthermore, since the $G$ action commutes with the $R$ action, the map $\rho$ restricts to a map $\rho:U(r\mathfrak{g}) \to \mathcal{O}(G) \uset{R} U(r \mathfrak{g})$.

Let $L$ be a closed subgroup of $G$. Then $L$ also acts on $U(r\mathfrak{g})$ via the restriction to $L$ of the Adjoint action of $G$. Again, by duality we obtain a comodule map $\rho_{r\mathfrak{g},L}:U(r\mathfrak{g}) \to \mathcal{O}(L) \uset{R} U(r\mathfrak{g})$.

Let $M$ be a $U(r\mathfrak{g})$-module that is also an $\mathcal{O}(L)$-comodule. The comodule structure induces an action of $L$;  the derivative of the $L$-action induces an action of the Lie algebra $\mathfrak{l}=\Lie(L)$, and so of $r \mathfrak{l}$, on $M$. Furthermore, since $U(r\mathfrak{g})$ and $M$ are $\mathcal{O}(L)$-comodules, so is $U(r\mathfrak{g}) \uset{R} M$, see \cite[Section 1.8]{Mont}  for details.

\begin{definition}
\label{d2classicaldefinitionofequivUmodule}
A weakly $L$-equivariant $U(r\mathfrak{g})$ module is a triple $(M,\alpha,\rho)$, where $M$ is a $R$-module, $\alpha:U(r\mathfrak{g}) \uset{R} M \to M$ is a left $U(r\mathfrak{g})$-action, $\rho:M \to \mathcal{O}(L) \uset{R} M$ is a $\mathcal{O}(L)$ co-action such that $\alpha$ is a morphism of $\mathcal{O}(L)$-comodules.

Furthermore, if the action of $r\mathfrak{l} \subset \mathfrak{l}=\Lie(L)$ induced by $\rho$ by derivating coincides with the restriction of the $r\mathfrak{g}$ action to $r\mathfrak{l}$, we say that $(M,\alpha,\rho)$ is $L$-equivariant. As for equivariant $\mathcal{D}$-modules, we will omit the equivariance structure when it is understood from the context.

A morphism of (weakly) $L$-equivariant $U(r\mathfrak{g})$-modules $(M,\alpha,\rho_1)$ and $(N,\beta,\rho_2)$ is a map $f:M \to N$ of Abelian groups that is $U(r\mathfrak{g})$-linear with respect to actions $\alpha,\beta$ and $\mathcal{O}(L)$-co-linear with respect to $\rho_1$ and $\rho_2$. We call such a morphism $L$-equivariant.

Denote $\Mod(U(r\mathfrak{g}),L)$ the category of consisting of $L$-equivariant $U(r\mathfrak{g})$-modules together with $L$-equivariant morphisms.

\end{definition}

We can reformulate the weakly equivariant condition in the following way: by the module-comodule correspondence $M$ can be viewed as a representation of the algebraic group $L$. Since $U(r\mathfrak{g})$ is also an $L$-representation we may rewrite the condition that the map $\alpha:U(r\mathfrak{g}) \uset{R} M \to M$ is a morphism of $\mathcal{O}(L)$-comodules as:

$$ l.( \psi.m)=(l.\psi).(l.m),$$

for all $R$-algebras $A$, $l \in L(A)$, $\psi \in U(r\mathfrak{g})_A$ and $m \in M_A$.  By abuse of language we define an equivalent notion of a weakly $L$-equivariant $U(r\mathfrak{g})$-module by:

\begin{center}
$M$ is a representation of $L$.
\end{center}
\begin{equation}
\label{d2equivliealgmodeq}
l.(\psi.m)=(l.\psi).(l.m) \text{ for all } l \in L, \psi \in U(r\mathfrak{g}), m \in M.
\end{equation} 

We will also need the following notion: let $\phi:U(r \mathfrak{g}) \to S$ be a map of rings. We say that an $S$-module is (weakly) $L$-equivariant if it (weakly) $L$-equivariant as $U(r \mathfrak{g})$-module. We denote $\Mod(S,L)$ the category of $L$-equivariant $S$-modules.

\subsection{Equivariance of two-sided ideals in classical enveloping algebras}

In this subsection we prove that any two-sided ideal in $U(\mathfrak{g})$ is a $G$-equivariant $U(\mathfrak{g} \times \mathfrak{g})$-module ($U(\mathfrak{g})- U(\mathfrak{g})$-bimodule) when the base ring is a field of characteristic 0. Here we view $G$ via its isomorphism with the diagonal subgroup of $G \times G$. The group $G$ acts on the enveloping algebra $U(\mathfrak{g})$ via the Adjoint action-denoted $\Ad$.


Let ${\tau}$ denote the principal anti-automorphism of $U(\mathfrak{g})$ induced by $x \mapsto -x$ for all $x \in \mathfrak{g}$. To simplify the notation, we will use $x^{\tau}$ to denote $\tau(x)$. We have an action of $U(\mathfrak{g} \times \mathfrak{g}) \cong U(\mathfrak{g}) \otimes U(\mathfrak{g})$ (by \cite[III.2.2)]{Ser}) on $U(\mathfrak{g})$ via

 $$ (x \otimes y) \cdot a= yax^{\tau} \text{ for all } x,a,y \in U(\mathfrak{g}).$$

\begin{proposition}
\label{d2gentwosidedidealequiv}
Assume that $R$ is a field of characteristic 0. Let $I$ be a two-sided ideal in $U(\mathfrak{g})$. Then $I \in \Mod(U(\mathfrak{g} \times \mathfrak{g}), G)$.

\end{proposition}

\begin{proof}
First, since $R$ is a field of characteristic $0$, we have  by \cite[Proposition 2.4.17]{Dix} that $I$ is invariant under the Adjoint action, so $I$ is a $G$-module. By construction, it is clear $I$ is a $U(\mathfrak{g}) \otimes U(\mathfrak{g})$-module under the action defined above. We then have that for all $g \in G$, $x,y \in U(\mathfrak{g}), u \in I$ that

\begin{equation}
\begin{split}
\Ad(g) \cdot ((x \otimes y) u) &= \Ad(g)(xuy^{\tau}) \\
                               &=\Ad(g)y \Ad(g)u \Ad(g)x^{\tau} \\
                               &=\Ad(g)y \Ad(g)u (\Ad(g)x)^{\tau} \\
                               &=(\Ad(g)x \otimes \Ad(g)y) \cdot \Ad(g)u.\end{split}
\end{equation}

The derivative of the $\Ad$ action is the $\ad$-action of the Lie Algebra $\mathfrak{g}$. Since $\Lie(G)$ embeds into $U(\mathfrak{g}) \otimes U(\mathfrak{g})$ via $x \to x \otimes 1 +1  \otimes x$ for $x \in \mathfrak{g}$, the two actions coincide. Therefore, we have proven all the conditions for $I$ to be a $G$-equivariant $U(\mathfrak{g} \times \mathfrak{g})$-module.
\end{proof}

We now specialise to two-sided ideals in $U(\mathfrak{g})$ with a given central character. For $\lambda \in \mathfrak{h}^* $, we let $\chi_{\lambda}:Z(\mathfrak{g}) \to R$ the associated central character. Let $m_{\lambda}=\ker(\chi_{\lambda})$ and $U(\mathfrak{g})^{\lambda}= U(\mathfrak{g})/ m_{\lambda} U(\mathfrak{g})$.

Consider the centre $C:=Z(U(\mathfrak{g} \times \mathfrak{g}))=Z(\mathfrak{g} \times \mathfrak{g})$. Since, $U(\mathfrak{g} \times \mathfrak{g}) \cong U(\mathfrak{g}) \otimes U(\mathfrak{g})$ and since $U(\mathfrak{g})$ is a free $R$-module one obtains $C \cong Z(\mathfrak{g}) \otimes Z(\mathfrak{g})$. Any central character of $C$ is determined by a pair $\theta_1,\theta_2$, where $\theta_i:Z(\mathfrak{g}) \to R$. Let $\lambda,\mu \in \mathfrak{h}^*$ and let $\theta_1=\ker_{\chi_{\lambda}}$ and $\theta_2=\ker_{\chi_{\mu}},$ so that we obtain: 

 $$\ker(\chi_{\lambda},\chi_{\mu})= m_{\lambda} \otimes Z(\mathfrak{g}) + Z(\mathfrak{g}) \otimes m_{\mu}.$$  

Therefore we obtain:

\begin{equation}
\label{d2quotientuguglambdalamba}
\begin{split}
U(\mathfrak{g} \times \mathfrak{g})^{\lambda,\mu} & \cong U(\mathfrak{g}) \otimes U(\mathfrak{g})/U(\mathfrak{g}) \otimes U(\mathfrak{g})  (m_{\lambda} \otimes Z(\mathfrak{g}) + Z(\mathfrak{g}) \otimes m_{\mu}) \\
           &\cong U(\mathfrak{g})/U(\mathfrak{g}) m_{\lambda} \otimes U(\mathfrak{g})/U(\mathfrak{g}) m_{\mu} \\
           &\cong U(\mathfrak{g})^{\lambda} \otimes U(\mathfrak{g})^{\mu}.
\end{split}
\end{equation}

Recall that if $I$ is a two-sided ideal in $U(\mathfrak{g})$ and $R$ a field of characteristic 0, we proved in Proposition \ref{d2gentwosidedidealequiv} that $I \in (U(\mathfrak{g} \times \mathfrak{g}),G)$. Furthermore, if $I$ is a two-sided ideal in  $U(\mathfrak{g})^{\lambda}$, we can view it as a $U(\mathfrak{g})^{\lambda}$-bimodule, so a module over the ring $U(\mathfrak{g})^{\lambda^{\op}} \otimes U(\mathfrak{g})^{\lambda}$. Further, we have by \cite[Lemma 5.4-Equation 5.5]{BeGi} that $\tau$ induces an isomorphism $U(\mathfrak{g})^{\lambda^{\op}} \cong U(\mathfrak{g})^{-w_{o} \lambda}$; recall that $w_{o}$ denotes the longest element of $W$. Therefore, using equation \eqref{d2quotientuguglambdalamba}, we deduce that a $U(\mathfrak{g})^{\lambda}$-bimodule is the same as $U(\mathfrak{g} \times \mathfrak{g})^{-w_{o} \lambda,\lambda }$-module. In particular, we obtain:
\begin{corollary}
\label{d2lambdaequivariancetwosided}
Assume $R$ is a field of characteristic 0. Let $I$ be a two-sided ideal in $U(\mathfrak{g})^{\lambda}$. Then $I \in \Mod(U(\mathfrak{g} \times \mathfrak{g})^{-w_{o} \lambda,\lambda},G)$.

\end{corollary}

\subsection{Equivariance of two-sided ideals in deformed enveloping algebras}

\textbf{Throughout this subsection only, we will assume that $R$ is a Noetherian local ring of characteristic 0.}

Let $r \in R$ a regular element and consider the $r$-th deformation $U(\mathfrak{g})_r \cong U(r \mathfrak{g})$. Unfortunately, it is not true that any two-sided ideal in $U(r \mathfrak{g})$ is $G$-equivariant with respect to the Adjoint action, so we restrict to a special class of ideals. We call a two-sided ideal $I$ in $U(r \mathfrak{g})$ an \emph{r-ideal} if $r^n i \in I$ for some $n \in \mathbb{N}$ and $i \in U(\mathfrak{g})$ implies that $i \in I$. This is equivalent to $U(r \mathfrak{g})/I$ having no $r$-torsion.

\begin{lemma}
\label{d2idealsinU(rg)closedunderadjointaction}
Let $I$ be an $r$-ideal in $U(r\mathfrak{g})$ and $x \in U(\mathfrak{g})$. Then $\ad_{x}(I) \subset I$. In other words, $I$ is closed under the adjoint action of $\mathfrak{g}$ on $U(\mathfrak{g})$.

\end{lemma}

\begin{proof}

Since $x \in U(\mathfrak{g})$, there exists $n \in \mathbb{N}$ such that $r^n x \in U(r\mathfrak{g})$. Since $I$ is an ideal in $U(r \mathfrak{g})$, we have $\ad_{r^n x}(I) \subset I$. The claim follows since $I$ is an $r$-ideal.
\end{proof}

For $\alpha \in \bm{\phi}$, we let $x_{\alpha}:G_a \to G$ and $e_{\alpha}=(dx_{\alpha})(1) \in \mathfrak{g}$ be the root homomorphism and root vector corresponding to $\alpha$, respectively.

\begin{corollary}
\label{d2ridealclosedunderAdjointactions}
Let $I$ be an $r$-ideal in $U(r\mathfrak{g})$. Then $\Ad(G(R))(I) \subset I$; in other words $I$ is closed under the Adjoint action. 

\end{corollary}

\begin{proof}

We have by \cite[Lemma 4.1b,c)]{Munster} that for all $\alpha \in \bm{\phi}$ and $s \in R$:

\begin{equation}
\label{d2Partofadjointactioneq}
x_{\alpha}(s) \cdot a= \sum_{i=0}^{\infty} \frac{ \ad(s e_{\alpha})^m}{m!} (a), 
\end{equation}
and there exists $n \in \mathbb{N}$ such that $\frac{ \ad(s e_{\alpha})^n}{n!} (a)=0$ for all $a \in U(\mathfrak{g})$. In particular, combining  Lemma \ref{d2idealsinU(rg)closedunderadjointaction} with equation \eqref{d2Partofadjointactioneq}, we obtain $x_{\alpha}(s) \cdot I \subset I$. To finish,  we have that $R$ is a local ring, so by \cite[Proposition 1.6]{Ab} the Chevalley group $G(R)$ is generated by elements of the form $x_{\alpha}(s)$.
\end{proof}

Recall that  $^{\tau}$ denote the principal anti-automorphism of $U(\mathfrak{g})$ induced by $x \mapsto -x$ for all $x \in \mathfrak{g}$. We have by combining \cite[Lemma 3.3]{KdimIwasawa} and the PBW theorem that $U(\mathfrak{g} \times \mathfrak{g})_r \cong U(r\mathfrak{g}) \otimes U(r\mathfrak{g})$ . We get an action of $ U(r\mathfrak{g}) \otimes U(r\mathfrak{g})$ on $U(r\mathfrak{g})$ via

 $$ (x \otimes y) \cdot a= yax^{\tau} \text{ for all } x,a,y \in U(r\mathfrak{g}).$$

\begin{proposition}
\label{d2r-idealareequiv}
Let $I$ be an $r$-ideal in $U(r\mathfrak{g})$. Then $I \in \Mod(U(\mathfrak{g} \times \mathfrak{g})_r, G)$.

\end{proposition}

\begin{proof}
We have by the Corollary above that $I$ is a $G$-module. By construction, it is clear $I$ is a $U(r\mathfrak{g}) \otimes U(r\mathfrak{g})$-module under the action defined above. We then have by the same arguments as in Proposition \ref{d2gentwosidedidealequiv} that for all $g \in G$, $x,y \in U(r\mathfrak{g}), u \in I$ 

$$\Ad(g) \cdot ((x \otimes y)u)=(\Ad(g)x \otimes \Ad(g)y) \cdot \Ad(g)u.$$

The derivative of the $\Ad$ is the $\ad$-action. Since $r \Lie(G)$ embeds into $U(r \mathfrak{g}) \otimes U(r \mathfrak{g})$ via $x \to x \otimes 1 + 1 \otimes  
x$ for $x \in r \mathfrak{g}$, the derivative of the $G$-action coincides with the Lie algebra action. This concludes the proof.
\end{proof}

\subsection{Verma modules}

An important tool in studying the representation theory of semisimple Lie Algebras are the Verma modules.

\begin{definition}

Let $\lambda \in r\mathfrak{h}^*$. We extend $\lambda$ to a map $r\mathfrak{n}^- \oplus r\mathfrak{h} \to R$ and to a ring morphism $\lambda:U(r\mathfrak{b}) \to R$ and denote $R_{\lambda}$ the corresponding $U(r\mathfrak{b})$-module. The Verma module of weight $\lambda$ is defined to be

         $$M(\lambda)= U(r\mathfrak{g}) \uset{U(r\mathfrak{b})} R_{\lambda}.$$

\end{definition}

For the rest of this subsection assume that $R=K$ is \emph{a field of characteristic $0$}. Given an $U(\mathfrak{g})$-module $M$ and a weight $\lambda\in \mathfrak{h}^*$, we denote 
$M_{\lambda}=\{m \in M| hm=\lambda(h)m\} \text{ for all } h \in \mathfrak{h}.$ Consider the BGG category $\mathcal{O}$ of finitely generated $U(\mathfrak{g})$-modules $M$ such that $M=\bigoplus_{\lambda \in \mathfrak{h}^*} M_{\lambda}$ and the action of $\mathfrak{n}$ on $M$ is locally finite.

We recall some basic facts about objects in category $\mathcal{O}$ that will be useful in the later sections.

\begin{proposition}
\label{d2vermammodulesproperties}
\begin{itemize}
\item The Verma module $M(\lambda)$ has a unique maximal submodule denoted $N(\lambda)$ and a unique simply quotient denoted $L(\lambda)$.
\item The annihilator of $M(\lambda)$ is given by $\ker (\chi_{\lambda}) U(\mathfrak{g})$.
\item Any module in category $\mathcal{O}$ has finite length.
\item The composition factors of objects in $\mathcal{O}$ are of the form $L(\mu)$ for $\mu \in \mathfrak{h}^*$.
\end{itemize}
\end{proposition}


\section{The localisation mechanism}
\label{d2sectionlocmec}

Throughout this section $G$ will denote a connected, simply-connected, split semisimple, smooth affine algebraic group over $R$, $\mathfrak{g}=\Lie(G)$ its Lie algebra, $X$ will denote an $R$-variety with a $G$-action and $r \in R$ a regular element.

\subsection{Equivariant localisation theory}

We fix $(\mathcal{D},i_{\mathfrak{g}})$ an $r$-deformed $G$-htdo on $X$. We aim to prove that even though Beilinson-Bernstein equivalence theorem does not work over an arbitrary commutative ring, the equivariance structure is preserved under localisation and taking global sections. Let $L$ be a closed subgroup of $G$; since $\mathcal{D}$ is a $G$-htdo, it is in particular an $L$-htdo and we denote by $\star$ the $L$-action on $\mathcal{D}$. The Lie algebra map $i_{\mathfrak{g}}: r \mathfrak{g} \to \mathcal{D}$ can be extended to a ring homomorphism $i_{\mathfrak{g}}:U(r \mathfrak{g}) \to \mathcal{D}$.

\begin{proposition}
\label{d2classicalloccomring}

Let $L$ be a closed subgroup of $G$ and $M$ an $L$-equivariant  $U(r\mathfrak{g})$-module. Then $\mathcal{D} \uset{U(r\mathfrak{g})} M$ is an $L$-equivariant quasi-coherent $\mathcal{D}$-module.  

\end{proposition}

To prove the proposition we will use two additional lemmas:

\begin{lemma}
\label{d2weaklyBBloc}

Let $M$ be weakly $L$-equivariant  $U(r\mathfrak{g})$-module. Then $\mathcal{D} \uset{U(r\mathfrak{g})} M$ is a weakly $L$-equivariant quasi-coherent $\mathcal{D}$-module.

\end{lemma}

\begin{proof}

We follow the idea in \cite[Section 10.4]{GRT}. Consider the induced $G$-actions on $\mathcal{D}$ and $M$ induced by equivariance condition. We define a twisted $G$-action on $\mathcal{D} \uset{R} M$ by viewing $M$ as a constant sheaf on $X$ and defining:

$$ l \cdot( D \otimes m)= l \star D \otimes l.m, $$

for any $D \in \mathcal{D}$, $m \in M$ and $l \in L$.

We begin by proving that the $\mathcal{D}$-module $\mathcal{D} \uset{R} M$ is weakly $L$-equivariant for the action defined above. We have:

\begin{equation}
\begin{split}
l \cdot D_1(D_2 \otimes m)&=  l \cdot (D_1D_2 \otimes m) \\
                            &=  l \star(D_1D_2) \otimes l.m \\
                            &= (l \star D_1) (l \star D_2) \otimes l.m \text { (by \ref{d2htdodef} i))} \\
                            &= (l \star D_1) ( l \cdot(D_2 \otimes m)), 
\end{split}
\end{equation}

for all $l \in L$, $D_1,D_2 \in \mathcal{D}, m \in M$.

 To prove that the $L$-action is well-defined on $\mathcal{D} \uset{U(r\mathfrak{g})} M$ (and thus  $\mathcal{D} \uset{U(r\mathfrak{g})} M$ is also weakly $L$-equivariant) it remains to prove that for any $\psi \in U(r\mathfrak{g})$, $m \in M$, $l \in L$ and $D \in \mathcal{D}$ a local section that

$$ l \cdot( D i_{\mathfrak{g}}(\psi) \otimes m)= l \cdot( D \otimes \psi.m).$$ 

We have

\begin{equation}
\begin{split}
l \cdot( D \otimes \psi.m) &= l \star D \otimes l.(\psi.m) \\
     &= l\star D \otimes (\Ad(l) \psi). (l.m) \text{ (by equation \eqref{d2equivliealgmodeq})} \\
     &=(l \star D) i_{\mathfrak{g}}(\Ad(l) \psi) \otimes l.m \\
     &=(l \star D) (l \star i_{\mathfrak{g}}(\psi)) \otimes l.m  \text{ (by \ref{d2htdodef} iii))}\\
     &=l \star (D i_{\mathfrak{g}}\psi) \otimes l.m \text{ (by \ref{d2htdodef} i))} \\
     &= l \cdot( D i_{\mathfrak{g}}(\psi) \otimes m). \qedhere
\end{split}
\end{equation}
\end{proof}

\begin{lemma}
\label{d2diffBB}

Let $L$ a closed subgroup of $G$ and let $M$ a $U(r\mathfrak{g})$-module that is also is an $\mathcal{O}(L)$-comodule and assume that the action of $r\mathfrak{l} \subset \mathfrak{l}=\Lie(L)$ induced by derivating the $L$-action coincides with the restriction of the $r\mathfrak{g}$ action to $r\mathfrak{l}$.

Consider the restriction to $r\mathfrak{l}$ of the $r\mathfrak{g}$-action on $\mathcal{D} \uset{U(\mathfrak{g})} M$ induced by $i_{\mathfrak{g}}:r\mathfrak{g} \to \mathcal{D}$ and the action of $r\mathfrak{l}$ induced by the derivative of the $L$ action on $\mathcal{D} \uset{U(\mathfrak{g})} M$. Then these two actions coincide.

\end{lemma}

\begin{proof}

Fix $\psi \in r\mathfrak{l}$, $D \in \mathcal{D}_X$ a local section and $m \in M$. 

The first action is given by $ \psi \cdot ( D \otimes m)= i_{\mathfrak{g}}(\psi) D \otimes m$.

Let $\star_1$ be the derivative of the $L$-action on $\mathcal{D}$ and $\star_2$ the derivative of the $L$-action on $M$. Then the second action (denote it $\star$) is given applying the chain rule by

$$ \psi \star (D \otimes m) = \psi \star_1 D \otimes m +  D \otimes \psi \star_2 m.$$

We have by Definition \ref{d2htdodef} iv) that $\psi \star_1 D= i_{\mathfrak{g}}(\psi)D- D i_{\mathfrak{g}}(\psi)$ and because of the assumption on $M$, $\psi \star_2 m= \psi m$. Therefore, we get

\begin{equation}
\begin{split}
\psi \star (D \otimes m) &=\psi \star_1 D \otimes m +  D \otimes \psi \star_2 m \\
        &= [i_{\mathfrak{g}}(\psi)D- D i_{\mathfrak{g}}(\psi)] \otimes m + D \otimes \psi m \\
        &= i_{\mathfrak{g}}(\psi)D  \otimes m- D i_{\mathfrak{g}}(\psi) \otimes m +  D \otimes \psi m \\
        &= i_{\mathfrak{g}}(\psi)D  \otimes m -  D \otimes \psi m + D \otimes \psi m \\
        &= i_{\mathfrak{g}}(\psi)D  \otimes m \\
        &= \psi \cdot (D  \otimes m).
\end{split}
\end{equation}

Thus the lemma is proved.
\end{proof}

Proposition \ref{d2classicalloccomring} now follows from Lemmas \ref{d2weaklyBBloc} and \ref{d2diffBB}.

\begin{proposition}
\label{d2classicalglocomring}
Let $L$ be a closed subgroup of $G$ and $\mathcal{M}$ an $L$-equivariant quasi-coherent $\mathcal{D}$-module. Then $M:=\Gamma(X,\mathcal{M})$ is an $L$-equivariant $U(r\mathfrak{g})$-module.  

\end{proposition}

We will do this in two steps:

\begin{lemma}
\label{d2weaklyBBglo}
Let $L$ be a closed subgroup of $G$ and $\mathcal{M}$ a weakly $L$-equivariant quasi-coherent $\mathcal{D}$-module. Then $M:=\Gamma(X,\mathcal{M})$ is a weakly $L$-equivariant  $U(\mathfrak{g})$-module.  
\end{lemma}

\begin{proof}

First, notice that the $U(r\mathfrak{g})$-module structure on $M$ is given by 

$$ \psi.m=i_{\mathfrak{g}}(\psi).m, \qquad \text{ for } \psi \in U(r\mathfrak{g}), m \in M.$$


We have that for $l \in L$, $\psi \in U(r\mathfrak{g})$ and $m \in M$

\begin{equation}
\begin{split}
l .( \psi .m) &= l. (i_{\mathfrak{g}}(\psi).m) \\
              &=(l. i_{\mathfrak{g}}(\psi)).(l.m) \text{ (by \ref{d2equivarianthtdomoddef} ii))} \\
              &=i_{\mathfrak{g}}(l.\psi).(l.m) \text{ (by \ref{d2htdodef} iii))} \\ 
              &=(l.\psi).(l.m). \qedhere
\end{split}
\end{equation}
\end{proof}

\begin{lemma}
\label{d2diffBBglo}
Let $\mathcal{M}$ be weakly $L$-equivariant $\mathcal{D}$-module and assume that the $r \mathfrak{l} \subset \mathfrak{l}=\Lie(L)$ action induced by the derivative of the $L$-action coincides with the restriction to $r\mathfrak{l}$ of the $r\mathfrak{g}$ action on $\mathcal{M}$ induced by $i_{\mathfrak{g}}$. Then the same holds on $M:=\Gamma(X,\mathcal{M})$.
\end{lemma}

\begin{proof}
The lemma follows from the fact that $\mathcal{D}$ is $G$-htdo and definition \ref{d2equivarianthtdomoddef}, notice it is much easier to prove than the corresponding Lemma \ref{d2diffBB}.  
\end{proof}

Proposition \ref{d2classicalglocomring} now follows from Lemmas \ref{d2weaklyBBglo}, \ref{d2diffBBglo}.

\subsection{Homogeneous twisted differential operators on the flag variety}

In this subsection, we aim to prove that the $r$-th deformation of the sheaf of  $\lambda$-twisted differential operators on the flag variety is an $r$-deformed $G$-htdo. We begin by reviewing the constructions in \cite[Section 4, Section 6.4]{Annals} and \cite[Section 5.1]{Eqdcap}. For now, we keep the notation from the previous section. 
Recall that by \cite[Definition 5.1.1]{Eqdcap} there is an infinitesimal action of $\mathfrak{g}$ on $X$ given by a $R$-linear Lie algebra map $\varphi': \mathfrak{g} \to \mathcal{T}_X(X)$. This morphism is $G$-equivariant by \cite[Lemma 5.1.3]{Eqdcap} and can be extended to a $G$-equivariant ring homomorphism $\alpha: U(\mathfrak{g}) \to \Gamma(X,\mathcal{D}_X)$. 

\begin{lemma}
\label{d2diffoperatorsexhtdo}
The sheaf $(\mathcal{D}_X,\alpha)$ is a $G$-htdo.
\end{lemma}

\begin{proof}

We have by \cite[Definition 4.2]{Sta1}  that $\mathcal{D}_X$ is a differential algebra with $F_0 \mathcal{D}_X \cong \mathcal{O}_X$ and $\gr_1 \mathcal{D}_X \cong \mathcal{T}_X$. Further by \cite[Equation 15]{Sta1}, we have $\gr \mathcal{D}_X \cong \Sym_{\mathcal{O}_X} \mathcal{T}_X$, so $\mathcal{D}_X$ is a tdo on $X$.

We endow $\mathcal{D}_X$ with a $G$ action by setting:

\begin{equation}
\begin{split}
&g.f(x)=f(gx) \text{ for } g \in G, f \in \mathcal{O}, x \in X, \\
&g.\tau(f)= g. \tau(g^{-1} f) \text{ for } g \in G, \tau \in \mathcal{T}, f \in \mathcal{O}.
\end{split}
\end{equation}

With this action, $\mathcal{D}_X$ satisfies axioms $i)$ and $ii)$ from Definition \ref{d2htdodef}. Since $\alpha$ is $G$-equivariant, axiom $iii)$ is also satisfied. Further, it is clear that the map $\alpha$ satisfies axioms $v)$ and $vi)$. Finally, axiom $iv)$ follows from the proof of \cite[Proposition 2.2]{MvdB}.
\end{proof}

\begin{lemma}
\label{d2imageofdiffeq}
Let $H$ be another smooth affine algebraic group acting on $X$ such that the actions of $G$ and $H$ commute. Then $\im(\alpha) \subset \Gamma(X,\mathcal{D}_X)^{H}$. 

\end{lemma}

\begin{proof}

This follows from \cite[Section 4.8]{Annals}.
\end{proof}

Recall that we assume that $G$ is a connected, simply connected, split semisimple, smooth affine algebraic group scheme over $R$. Let $B$ be a closed and flat Borel $R$-subgroup scheme, $N$ its unipotent radical and $H=B/N$ the abstract Cartan group. Let $\widetilde{X}=G/N $ denote the basic affine space and $X=G/B$ denote the flag scheme. Define an action of $H$ on $\widetilde{X}$ via 
 
       $$bN \cdot gN:=gbN, \qquad b \in B, g \in G.$$

\begin{lemma}[{\cite[Lemma 4.7]{Annals}}]
\label{d2Hloctriviallytorsor}
\begin{enumerate}[label=\roman*)]
\item{The action of $H$ commutes with the natural action of $G$.}
\item{$\widetilde{X}$ and $X$ are smooth separated schemes over $R$, locally of finite type.}
\item{The natural projection $\xi:\widetilde{X}\to X$ is a locally trivial $H$-torsor.}
\end{enumerate}

\end{lemma}

\begin{definition}
The relative enveloping algebra is the sheaf of $H$-invariants of $\xi_* \mathcal{D}_{\widetilde{X}}$:

           $$\widetilde{\mathcal{D}}_X:=(\xi_* \mathcal{D}_{\widetilde{X}})^H.$$

\end{definition}

The sheaf comes with a natural filtration $F_i \widetilde{\mathcal{D}}_X:= (\xi_* F_i \mathcal{D}_{\widetilde{X}})^H$ induced by the natural filtration on  $\mathcal{D}_{\widetilde{X}}$.

\begin{proposition}
\label{d2basicpropdtilde}
Let $\widetilde{\mathcal{T}}_X:=(\xi_* \mathcal{T}_{\widetilde{X}})^H$, $\mathfrak{h}=\Lie(H)$, $U$ be an affine open set of $X$ trivialising $\xi$. Then:

\begin{enumerate}[label=\roman*)]
\item{$\mathcal{D}_{\widetilde{X}}(U) \uset{R} U(\mathfrak{h}) \cong  \widetilde{\mathcal{D}}_X(U).$}
\item{$\Sym_{\mathcal{O}_X} \widetilde{\mathcal{T}}_X \cong \gr \widetilde{\mathcal{D}}_X$.}
\item{The derivative of the $H$ action on $\widetilde{X}$ induces a central embedding $j: \mathfrak{h} \to \widetilde{\mathcal{D}}_X$. }
\end{enumerate}

\end{proposition}

\begin{proof}
The first two claims follow by \cite[Proposition 4.6]{Annals} and the third follows by \cite[Section 4.10]{Annals} along with Lemma \ref{d2Hloctriviallytorsor} $ii)$ and $iii)$.
\end{proof}

The sheaf $\widetilde{\mathcal{D}}_X$ will be used to construct the sheaf of $\lambda$-twisted differential operators. Proposition \ref{d2basicpropdtilde} $i)$ shows that $\widetilde{\mathcal{D}}_X$ can not be a htdo as it is not a tdo on $X$.

\begin{lemma}
\label{d2Dtildeaxioms}
The sheaf $(\widetilde{\mathcal{D}}_X,\alpha)$ is a differential algebra, $G$-equivariant as an $\mathcal{O}_X$-module and furthermore it satisfies axioms $i)$ to $vi)$ of Definition \ref{d2htdodef}.
\end{lemma}

\begin{proof}
The filtration on $\widetilde{\mathcal{D}}_X$ makes it a differential algebra. Since the actions of $G$ and $H$ commute we have by \cite[Lemma 3.6]{Sta1} that the $G$ action on $\mathcal{D}_{\widetilde{X}}$ descends to a $G$-action on $\widetilde{\mathcal{D}}_X$. Thus, axioms $i),ii)$ of Definition \ref{d2htdodef} are satisfied. Axioms $iii)-vi)$ follow from construction and Lemma \ref{d2imageofdiffeq} since $\alpha$ and $\xi$ are $G$-equivariant.
\end{proof}

\begin{definition}
Let $\widetilde{\mathcal{D}}_r$ be the sheafification of the presheaf obtained by postcomposing $\widetilde{\mathcal{D}}_X$ with the deformation functor $A \to  A_r$. Since the $G$-action preserves the filtration on $\widetilde{\mathcal{D}}_X$ induced by deformation, it restricts to a $G$-action on $\widetilde{\mathcal{D}}_r$.
\end{definition}

Now, let $\lambda: r\mathfrak{h} \to R$; the map $\lambda$ can be extended to a ring homomorphism $\lambda:U(r\mathfrak{h}) \to R$ and we denote $R_{\lambda}$ the resulting $U(r\mathfrak{h})$-module. Furthermore, recall that by Proposition \ref{d2basicpropdtilde} $iii)$, there exists a central embedding $j: \mathfrak{h} \to \widetilde{\mathcal{D}}_X$ which can be extended to a ring morphism $j: U(\mathfrak{h}) \to \widetilde{\mathcal{D}}_X$. By applying the deformation functor we obtain a homomorphism $j:U(r\mathfrak{h}) \to \widetilde{\mathcal{D}}_r$.

\begin{definition}
\label{d2Dlambdardefinition}

The sheaf of $r$-deformed $\lambda$-twisted differential operators on the flag variety $X$ is the central reduction 

         $$\mathcal{D}_{\lambda,r}:=\widetilde{\mathcal{D}}_r \uset{U(r\mathfrak{h})} R_{\lambda}.$$

We give $R_{\lambda}$ the trivial filtration and we view $\mathcal{D}_{\lambda,r}$ as a sheaf of filtered $R$-algebras with the tensor filtration.

\end{definition}

\begin{corollary}
\label{d2Dlambdaishtdo}

The sheaf $(\mathcal{D}_{\lambda,r},\alpha)$ is an $r$-deformed $G$-htdo on the flag variety.

\end{corollary}

\begin{proof}

We have by \cite[Lemma 6.4]{Annals} that $\mathcal{D}_{\lambda}$ is a tdo on $X$. Since $\mathcal{D}_{\lambda,r}$ can be viewed as the $r$-th deformation of $\mathcal{D}_{\lambda}$, we have by Lemma \ref{d2deformationoftdoisrtdo} that $\mathcal{D}_{\lambda,r}$ is an $r$-deformed tdo. 
Let $\mathcal{I}=\langle j(h)-\lambda(h) | h \in r\mathfrak{h} \rangle$ be two-sided ideal in $\widetilde{\mathcal{D}}_r$ such that $\mathcal{D}_{\lambda,r}=\widetilde{\mathcal{D}}_{r} / \mathcal{I}$. We have by Lemma \ref{d2imageofdiffeq} (note that we swap $G$ and $H$) that $g.j(h)=h$. Therefore, the ideal $\mathcal{I}$ is stable under the $G$-action, so the $G$-action on $\widetilde{\mathcal{D}}_r$ descends to $\mathcal{D}_{\lambda,r}$. By abuse of notation we denote $\alpha$ the composition $\alpha:U(r\mathfrak{g}) \to \widetilde{\mathcal{D}}_r$ with the projection $\widetilde{\mathcal{D}}_r \to \mathcal{D}_{\lambda,r}$. Then $(\mathcal{D}_{\lambda,r},\alpha)$ is an $r$-deformed $G$-htdo.
\end{proof}

We finish the subsection by computing the geometric fibre of $\mathcal{D}_{\lambda}=\mathcal{D}_{\lambda,1}$ at the identity. This is also sketched is \cite[I.2.4]{Mil2}, but we use different conventions to construct the sheaf $\mathcal{D}_{\lambda}$.  Since the group $G$ is split we can find a Cartan group $T$ of $G$ complementary to $N$ in $B$. Using the natural isomorphism $H \to T$, we view $\mathfrak{h}$ as a Cartan subalgebra of $\mathfrak{g}$. We let $\mathfrak{g}= \mathfrak{n^{-}} \oplus \mathfrak{h} \oplus \mathfrak{n}$ and $\mathfrak{b}= \mathfrak{h} \oplus \mathfrak{n}$ to be the corresponding decompositions for the Lie algebra and Borel subalgebra respectively. The adjoint action of $H$ on $\mathfrak{g}$ induces a decomposition $\mathfrak{g}=\mathfrak{n}^{-} \oplus \mathfrak{h} \oplus \mathfrak{n}^{+}$, where we regard $\mathfrak{n}^{-}=\Lie(N)$ as being spanned by \emph{negative} roots.

\begin{lemma}
\label{d2geometricfibreDtilde}
Let $m$ be the ideal sheaf of functions on $X$ vanishing on $eB$ and let $i:[eB] \to X$ denote the natural inclusion. Then:

          $$U(\mathfrak{g})/\mathfrak{n}^{-} U(\mathfrak{g}) \cong \Gamma(X,\widetilde{\mathcal{D}}_X/m \widetilde{\mathcal{D}}_X).$$

\end{lemma}

\begin{proof}
Let $Y=G/N$ denote the basic affine space, let $x_0=eN$ denote the base point and $m_N$ denote the ideal sheaf of functions vanishing on $x_0$. 

Let $\rho:\mathfrak{g} \to G/N$ encode the infinitesimal action of $\mathfrak{g}$ on $G/N$ induce by the action of $G$ on $G/N$ by left multiplication. Since the stabiliser of $x_0$ with respect to this action is $N$ we have that:

$$ \rho(\mathfrak{n}^{-}) \subset m_N \mathcal{D}_{G/N}.$$

Applying the descent functor $(\xi_*)^H$ we obtain a map $\tilde{\rho}:\mathfrak{n}^{-} \to (m_N \mathcal{D}_{G/N})^H \cong (m_N)^H \widetilde{\mathcal{D}}_X$.  By construction $\tilde{\rho}$ is just the restriction of the infinitesimal action of $\mathfrak{g}$ on $G/B$ to $\mathfrak{n}^{-}$. By considering the following diagram

\begin{center}
\begin{tikzcd}

& G/N \arrow[d]  & x_0=N/N \arrow[l] \arrow[d] \\
& G/B   & eB  \arrow[l]
\end{tikzcd}
\end{center}

we observe that $(m_N)^H \widetilde{\mathcal{D}}_X \subset m \widetilde{\mathcal{D}}_X$, so $\tilde{\rho}(\mathfrak{n}^{-}) \subset m \widetilde{\mathcal{D}}_X$. Therefore we obtain a surjective map 

$$ U(\mathfrak{g})/\mathfrak{n}^{-} U(\mathfrak{g}) \to \Gamma(X,\widetilde{\mathcal{D}}_X/m \widetilde{\mathcal{D}}_X).$$

By passing to the associated grading rings and applying Proposition \ref{d2basicpropdtilde} $ii)$ we obtain an isomorphism $\gr( U(\mathfrak{g})/\mathfrak{n}^{-} U(\mathfrak{g})) \to \Gamma(X,\gr (\widetilde{\mathcal{D}}_X/m \widetilde{\mathcal{D}}_X))$, so the claim follows.

\end{proof}

\begin{corollary}
\label{d2geometricfibreDlambda}
Let $\lambda \in \mathfrak{h}^*$ and let $m$ the ideal sheaf of functions on $X$ vanishing on $eB$. Further, let $R_{\lambda}$ be the $U(\mathfrak{h})$-module induced by the $\mathfrak{h}$ module on which $\mathfrak{h}$ acts by $\lambda$. Then

          $$R_{\lambda} \uset{U(\mathfrak{h})} (U(\mathfrak{g})/\mathfrak{n}^{-}U(\mathfrak{g})) \cong \Gamma(X,\mathcal{D}_{\lambda}/m \mathcal{D}_{\lambda}). $$

\end{corollary}

\begin{proof}

We have:

\begin{equation}
\begin{split}
\Gamma(X, \mathcal{D}_{\lambda}/m \mathcal{D}_{\lambda})&= \Gamma(X, (i_*) \mathcal{O}_{eB} \uset{\mathcal{O}_X} \widetilde{\mathcal{D}}_X) \uset{U(\mathfrak{h})} R_{\lambda} \\
&\cong \Gamma(X, \widetilde{\mathcal{D}}_X)/m \widetilde{\mathcal{D}}_X) \uset{U(\mathfrak{h})} R_{\lambda})\\
&\cong \Gamma(X, R_{\lambda} \uset{U(\mathfrak{h})} \widetilde{\mathcal{D}}_X)/m \widetilde{\mathcal{D}}_X) \text{ (by Proposition \ref{d2basicpropdtilde})} \\
& \cong R_{\lambda} \uset{U(\mathfrak{h})} \Gamma(X, \widetilde{\mathcal{D}}_X)/m \widetilde{\mathcal{D}}_X) \\
& \cong R_{\lambda} \uset{U(\mathfrak{h})} U(\mathfrak{g})/\mathfrak{n}^{-} U(\mathfrak{g}) \text{ (by Lemma \ref{d2geometricfibreDtilde})}.
\end{split}
\end{equation}
\end{proof}

\subsection{Applications of the localisation mechanism}

We retain the notation from the previous subsection. Recall that for any weight $\lambda: \mathfrak{h} \to R$, we defined $\chi_{\lambda}$ the corresponding central character and we let $U(\mathfrak{g})^{\lambda}= U(\mathfrak{g})/ \ker(\chi_{\lambda}) U(\mathfrak{g})$.

By \cite[Section 6.10]{Annals}, we have a map $\varphi:U(\mathfrak{g})^{\lambda} \to \mathcal{D}_{\lambda}$; by functoriality of deformation, we obtain a map $\varphi:U(\mathfrak{g})_r^{\lambda} \to \mathcal{D}_{\lambda,r}$. Define the localisation functor $\Loc^{\lambda,r}:\Mod(U(\mathfrak{g})_r^{\lambda}) \to \Mod(\mathcal{D}_{\lambda,r}$). This functor is adjoint to the global sections functor $\Gamma(X,-): \Mod(\mathcal{D}_{\lambda,r}) \to \Mod(U(\mathfrak{g})_r^{\lambda}).$

\begin{proposition}

Let $L$ be a closed subgroup of $G$. The functors $\Loc^{\lambda,r}$ and $\Gamma(X,-)$ induce a pair of adjoint functors between $\Mod(U(\mathfrak{g})_r^{\lambda},L)$ and $\Mod(\mathcal{D}_{\lambda,r},L)$.

\end{proposition}

\begin{proof}

This follows by combining Corollary \ref{d2Dlambdaishtdo} and Propositions \ref{d2classicalloccomring} and \ref{d2classicalglocomring}. 
\end{proof}

In general, the functors above do not provide an equivalence of categories. However, when $R$ is a field of characteristic $0$ we get the well-known Beilinson-Bernstein localisation. We will also need the following proposition:

\begin{proposition}
\label{d2envelopingalgebratoDlambda}
Assume that $R$ is a field of characteristic $0$. Then the map $\varphi:U(\mathfrak{g})^{\lambda} \to \Gamma(X,\mathcal{D}_{\lambda})$ is an isomorphism.

\end{proposition}

\begin{theorem} (\cite{BB}, \cite[Section 9.2]{Kashi})
\label{d2equivariantbb}
Assume that $R$ is a field of characteristic $0$. Let $\lambda$ be a dominant weight and  $L$  a closed subgroup of $G$.  Then the functor $\Loc^{\lambda}$ is an equivalence of categories between $\Mod(U(\mathfrak{g})^{\lambda},L)$ and the quotient category  $\Mod(\mathcal{D}_{\lambda},L)/ \ker \Gamma$. A quasi-inverse is given by $\Gamma(X,-)$. 

Furthermore, under this equivalence finitely generated $U(\mathfrak{g})^{\lambda}$-modules correspond to coherent $\mathcal{D}_{\lambda}$-modules and the global sections functor is exact.

If $\lambda$ is also regular, $\ker \Gamma=0$, so one obtains an equivalence of categories between $\Mod(U(\mathfrak{g})^{\lambda},L)$ and $\Mod(\mathcal{D}_{\lambda},L)$. 

The same statements hold if we remove the $L$-equivariance from both sides, i.e. setting $L=e$.

\end{theorem}

As an easy consequence we obtain:

\begin{corollary}
\label{d2glolocisomor}
Assume that $R$ is a field of characteristic 0 and let $\lambda$ be a dominant weight. Consider $M \in \Mod(U(\mathfrak{g})^{\lambda},L)$. Then 

                    $$\Gamma(X,\Loc^{\lambda}M) \cong M.$$

\end{corollary}

In the next sections, we will apply the theorem in two particular cases: $B$ is a Borel subgroup of $G$ acting on $X=G/B$ and $G=\{(g,g)| g \in G\} \subset G \times G$ acting diagonally on the flag variety of $G \times G$, which is $X \times X$. Notice, we are using a slight abuse notation and identify $G$ with the diagonal subgroup.


\section{Pullback of modules over deformed \texorpdfstring{\\ homogeneous}{homogeneous} twisted differential operators}
\label{d2sectionpullbackfromdoubleflagvariety}

We retain the notation from the beginning of the previous section: recall that  $G$ denotes a connected, simply-connected, split semisimple, smooth affine algebraic group scheme over $R$, $\mathfrak{g}=\Lie(G)$, $B$ a closed and flat Borel subscheme of $G$ and $X=G/B$ the flag scheme. We consider the triangular decomposition $\mathfrak{g}=\mathfrak{n}^{-} \oplus \mathfrak{h} \oplus \mathfrak{n}^+$, we let $\mathfrak{b}^{-}:=\mathfrak{n}^{-} \oplus \mathfrak{h}$, $\mathfrak{b}:=\mathfrak{h} \oplus \mathfrak{n}^+$.

By construction, $X$ is a $G$-homogeneous space; furthermore we have by \cite[II,1.10(2)]{Jan1} and by Lemma \ref{d2Hloctriviallytorsor} $ii)$ that Assumption \ref{d2loctrivialassumption} is satisfied. Therefore, we may apply the machinery developed in Sections  \ref{d2sectionBoBr} and \ref{d2sectionlocmec}.

We consider homogeneous twisted differential operators on the double flag variety $X \times X$ of the group $G \times G$. The Cartan algebra of $r\mathfrak{g} \times r\mathfrak{g}$ is given by $r\mathfrak{h} \times r\mathfrak{h}$, so picking a weight of the deformed Cartan Lie subalgebra is equivalent to picking a pair of weights $\lambda,\mu \in r\mathfrak{h}^*$. We have by Corollary \ref{d2Dlambdaishtdo} that $\mathcal{D}_{\lambda,\mu,r}$ is an $r$-deformed $G \times G$-htdo on $X \times X$. In particular,  it is an $r$-deformed $G$-htdo with respect to the diagonal $G$ action defined in Section \ref{d2sectionBoBr}. Recall also that $i_l:X \to X \times X$ denotes the inclusion into the left copy. We have by Corollary \ref{d2leftinclusionalaBorho-Brylinski} that $\mathcal{D}_{\lambda,r} \cong i_l^{\#} \mathcal{D}_{\lambda,\mu,r}$ is an $r$-deformed $B$-htdo and the functor $i_l^{\#}: \Coh(\mathcal{D}_{\lambda,\mu,r},G) \to \Coh(\mathcal{D}_{\lambda,r},B)$ is an equivalence of categories.
From now on, until the end of this section, we will assume that $\lambda,\mu:r\mathfrak{h} \to R$ are dominant weights. 

\begin{definition}
Let $i:eB \to X$ be the natural inclusion and let $R$ to be trivial $\mathcal{O}_{eB}$-module. We let $\mathcal{M}_{\mu,r}$  be the right $\mathcal{D}_{\mu,r}$-module defined by:

                   $$\mathcal{M}_{\mu,r}:=i_* R \uset{\mathcal{O}_X} \mathcal{D}_{\mu,r}.$$

\end{definition}

\begin{lemma}
\label{d2Mlambdaproperties}
\leavevmode
\begin{enumerate}[label=\roman*)]
\item{$\mathcal{M}_{\mu,r}$ is coherent as a right $\mathcal{D}_{\mu,r}$-module.}
\item{Let $\mathcal{M}_{\mu}=\mathcal{M}_{\mu,1}$. Then $\Gamma(X,\mathcal{M}_{\mu})=R_{\mu} \uset{U(\mathfrak{b}^{-})} U(\mathfrak{g})$; we will use $T(\mu)$ to denote this module.}
\end{enumerate}
\end{lemma}

\begin{proof}

Since $\mathcal{D}_{\mu,r}$ is in particular an $r$-deformed tdo, we have by Definition \ref{d2tdodefinition} that $\gr \mathcal{D}_{\mu,r} \cong \Sym_{\mathcal{O}_X}(r \mathcal{T}_X)$ is a sheaf of Noetherian rings, so $\mathcal{D}_{\mu,r}$ is a sheaf of Noetherian rings. Therefore, coherence is equivalent to locally finite generation. Let $U \subset X$ be affine open; if $eB \notin U$, then $\mathcal{M}_{\mu,r}(U)=0$. If $eB \in U$, then $1 \otimes 1$ is a generator for $\mathcal{M}_{\mu,r}(U)$.

For the second part, we have by construction that $\mathcal{M}_{\mu} \cong i_*(i^* \mathcal{D}_{\mu}) \cong T_{eB} (\mathcal{D}_{\mu})$ (the geometric fibre at the identity). Let $m$ be the ideal sheaf of functions on $X$ vanishing on $eB$. We have:

\begin{equation}
\begin{split}
 T_{eB} (\mathcal{D}_{\mu}) &= \Gamma(X,\mathcal{D}_{\mu}/ m \mathcal{D}_{\mu}) \\
                                &=R_{\mu} \uset{U(\mathfrak{h})}  U(\mathfrak{g})/\mathfrak{n}^{-}U(\mathfrak{g}) \text { (by Corollary \ref{d2geometricfibreDlambda}) } \\
                                &=R_{\mu} \uset{U(\mathfrak{b}^{-})} U(\mathfrak{g}). \qedhere
\end{split}
\end{equation}
\end{proof}

\begin{lemma}
\label{d2pullbackofDlambdalambda}
Let $p_r:X \times X \to X$ denote the projection onto the right factor. Then

$$p_r^{-1}(\mathcal{M}_{\mu,r}) \uset{p_r^{-1} \mathcal{D}_{\mu,r}} \mathcal{D}_{\lambda,\mu,r} \cong i_{l_*} i_l^{*} \mathcal{D}_{\lambda,\mu,r}.$$

\end{lemma}

\begin{proof}
\label{d2pullbackunderiltdo}

We have

\begin{equation}
\begin{split}
i_{l_*}i_l^* \mathcal{D}_{\lambda,\mu,r} &=i_{l_*}(\mathcal{O}_X \uset{i_l^{-1} \mathcal{O}_{X \times X}} i_l^{-1} \mathcal{D}_{\lambda,\mu,r}) \\
                &\cong i_{l_*} \mathcal{O}_X \uset{\mathcal{O}_{X \times X}} \mathcal{D}_{\lambda, \mu,r} \\
                & \cong i_{l_*} \mathcal{O}_X \uset{\mathcal{O}_{X \times X}} p_r^{-1} \mathcal{D}_{\mu,r} \uset{p_r^{-1} \mathcal{D}_{\mu,r}} \mathcal{D}_{\lambda,\mu,r}.
\end{split}
\end{equation}

Therefore, it is enough to prove that
              $$p_r^{-1} \mathcal{M}_{\mu,r} \cong  i_{l_*} \mathcal{O}_X \uset{\mathcal{O}_{X \times X}} p_r^{-1} \mathcal{D}_{\mu,r}. $$

Consider the following Cartesian square:

\begin{center}
\begin{tikzcd}

&X \arrow[r,"p"] \arrow[d,"i_l"] &eB \arrow[d,"i"] \\
&X \times X \arrow[r,"p_r"]   &X, 
\end{tikzcd}
\end{center}

where $i$ and $p$ denote the natural inclusion and projection. Consider the constant sheaf $R$ on the trivial scheme $eB$. We then have by \cite[\href{https://stacks.math.columbia.edu/tag/02KG}{02KG}]{StackProject} 

 $$p_r^*(i_*R) \cong i_{l_*} (p^* R).$$   
              
Here the pullback is taken in the category of right modules instead of left modules. Since $p^* R \cong \mathcal{O}_X$ as right $\mathcal{O}_X$-modules, we obtain that as right $\mathcal{O}_{X \times X}$-modules

 $$i_{l_*} \mathcal{O}_X \cong p_r^* (i_* R).$$
 
Therefore, we obtain

\begin{equation}
\begin{split}
 i_{l_*} \mathcal{O}_X \uset{\mathcal{O}_{X \times X}} p_r^{-1} \mathcal{D}_{\mu,r} & \cong p_r^* (i_* R) \uset{\mathcal{O}_{X \times X}} p_r^{-1} \mathcal{D}_{\mu,r} \\
                   & \cong p_r^{-1}(i_* R) \uset{p_r^{-1} \mathcal{O}_X} \mathcal{O}_{X \times X} \uset{\mathcal{O}_{X \times X}} p_r^{-1} \mathcal{D}_{\mu,r} \\
                   & \cong p_r^{-1}(i_* R) \uset{p_r^{-1} \mathcal{O}_X} p_r^{-1} \mathcal{D}_{\mu,r} \\
                   & \cong p_r^{-1}(i_* R \uset{\mathcal{O}_X} \mathcal{D}_{\mu,r}) \\
                   & \cong p_r^{-1} \mathcal{M}_{\mu,r}. \qedhere
\end{split}
\end{equation} 
\end{proof}

\begin{corollary}
\label{d2twistedpullbacksimplified}
Let $\mathcal{M}$ be a coherent $\mathcal{D}_{\lambda,\mu,r}$-module. Then
$$i_{l_*}i_l^{\#} \mathcal{M} \cong p_r^{-1}(\mathcal{M}_{\mu,r}) \uset{p_r^{-1} \mathcal{D}_{\mu,r}} \mathcal{M}.$$

\end{corollary}

\begin{proof}

We have

\begin{equation}
\begin{split}
i_{l_*} i_l^{\#} \mathcal{M} &\cong i_{l_*}i_l^* \mathcal{M} \\
                             &\cong i_{l_*}( \mathcal{O}_X \uset{i_l^{-1} \mathcal{O}_{X \times X}} i_l^{-1} \mathcal{M}) \\
                             &\cong i_{l_*}( \mathcal{O}_X \uset{i_l^{-1} \mathcal{O}_{X \times X}} i_l^{-1} \mathcal{D}_{\lambda,\mu,r} \uset{i_l^{-1} \mathcal{D}_{\lambda,\mu,r}} i_l^{-1} \mathcal{O}_{X \times X} \uset{i_l^{-1} \mathcal{O}_{X \times X}} i_l^{-1} \mathcal{M})\\
&\cong i_{l_*} (i_l^*\mathcal{D}_{\lambda,\mu,r} \uset{i_l^{-1} \mathcal{D}_{\lambda,\mu,r}} i_l^{-1} \mathcal{M}) \\
          &\cong i_{l_*} i_l^*\mathcal{D}_{\lambda,\mu,r} \uset{\mathcal{D}_{\lambda,\mu,r}} \mathcal{M} \text{ (Since $i_l$ is closed)} \\
         &\cong p_r^{-1}(\mathcal{M}_{\mu,r}) \uset{p_r^{-1} \mathcal{D}_{\mu,r}} \mathcal{D}_{\lambda,\mu,r} \uset{\mathcal{D}_{\lambda,\mu,r}} \mathcal{M}  \text{ (By Lemma \ref{d2pullbackofDlambdalambda})} \\
         &\cong  p_r^{-1}(\mathcal{M}_{\mu,r}) \uset{p_r^{-1} \mathcal{D}_{\mu,r}} \mathcal{M}. \qedhere
\end{split}
\end{equation}
\end{proof}

\subsection{Global sections under the pullback}
\label{d2sectionglosec}

\textbf{Throughout this subsection, we will assume that $R=K$ is a field of characteristic $0$.}

Recall that by equation \eqref{d2quotientuguglambdalamba}, we have $U(\mathfrak{g} \times \mathfrak{g})^{\lambda,\mu} \cong U(\mathfrak{g})^{\lambda} \otimes U(\mathfrak{g})^{\mu}.$ We aim to prove the following theorem: 
\begin{theorem}
\label{d2globalsectionspullback}
Let $\lambda,\mu$ be dominant weights and let $\mathcal{M}$ be a coherent $G$-equivariant $\mathcal{D}_{\lambda,\mu}$-module. Then we obtain an isomorphism of left $U(\mathfrak{g})^{\lambda}$-modules
      $$\Gamma(X,i_l^{\#} \mathcal{M}) \cong T(\mu) \uset{U(\mathfrak{g})^{\mu}} \Gamma(X \times X, \mathcal{M}).$$

\end{theorem}

We should make several remarks before we proceed to prove this.  We view $\Gamma(X \times X,\mathcal{M})$ as a left $U(\mathfrak{g})^{\mu}$-module by defining $y.m=(1 \otimes y).m$ for $y \in U(\mathfrak{g})^{\mu}$ and $m \in \Gamma(X \times X, \mathcal{M})$. Further, we note that  the space  $T(\mu) \uset{U(\mathfrak{g})^{\mu}} \Gamma(X \times X, \mathcal{M})$ has the structure of a left $U(\mathfrak{g})^{\lambda}$-module via $x.(m \otimes m')=m \otimes (x \otimes 1).m'$ for $x \in U(\mathfrak{g})^{\lambda}$, $m \in T(\mu)$ and $m' \in \Gamma(X \times X,\mathcal{M})$.

In order to prove this theorem, we need to introduce additional notation. Let $Y,Z$ be $K$-schemes, $\mathcal{M}$ an $\mathcal{O}_Y$-module and $\mathcal{N}$ an $\mathcal{O}_Z$-module. We will use the notation:

\begin{equation}
\mathcal{M} \boxtimes \mathcal{N}:= \mathcal{O}_{Y \times Z} \uset{q_Y^{-1} \mathcal{O}_Y \otimes q_Z^{-1} \mathcal{O}_Z} q_Y^{-1} \mathcal{M} \otimes q_Z^{-1} \mathcal{N} \cong q_Y^* \mathcal{M} \uset{ \mathcal{O}_{Y \times Z}} q_Z^* \mathcal{N},
\end{equation}

where $q_Y:Y \times Z \to Y$ and $q_Z:Y \times Z \to Z$ denote the natural projections. We refer to \cite[Section I.1.5]{HTT} for the basic properties of box tensor product. We will also identify $X$ with $X \times eB$ and the map $i_l$ with $\id \times i$; recall $i:eB \to X$ denotes the natural inclusion. We have that:

\begin{equation}
\begin{split}
i_l^* \mathcal{D}_{\lambda,\mu} &\cong (\id \times i) \mathcal{D}_{\lambda,\mu} \\
&\cong (\id \times i) (\mathcal{D}_{\lambda} \boxtimes \mathcal{D}_{\mu}) \text{ (by \cite[Section 5.4]{Gin3})} \\
& \cong \mathcal{D}_{\lambda} \boxtimes i^* \mathcal{D}_{\mu} \\
& \cong \mathcal{D}_{\lambda} \uset{K} \Gamma(eB,i^* \mathcal{D}_{\mu}) \\
& \cong \mathcal{D}_{\lambda} \uset{K} \Gamma(X,i_* i^*\mathcal{D}_{\mu})  \\
& \cong \mathcal{D}_{\lambda} \uset{K} \Gamma(X,\mathcal{M}_{\mu}) \\
& \cong \mathcal{D}_{\lambda} \uset{K} T(\mu) \text{ (by Lemma \ref{d2Mlambdaproperties}).} 
\end{split}
\end{equation}

Therefore we obtain that as left $\mathcal{D}_\lambda$-modules $i_l^{\#} \mathcal{D}_{\lambda,\mu} \cong \mathcal{D}_{\lambda} \uset{K} T(\mu)$, so by Theorem \ref{d2equivariantbb}, we have that as left $U(\mathfrak{g})^{\lambda}$-modules:

\begin{equation}
\label{d2globalsectionDlambdalambdamtdo}
\Gamma(X,i_l^{\#} \mathcal{D}_{\lambda,\mu}) \cong U(\mathfrak{g})^{\lambda} \uset{K} T(\mu).
\end{equation}

\begin{proof}[Proof of theorem \ref{d2globalsectionspullback}]
Let $\mathscr{F},\mathscr{J}:\Coh(\mathcal{D}_{\lambda,\mu},G) \to \Mod(U(\mathfrak{g})^{\lambda},B)$ defined by

\begin{equation}
\begin{split}
 \mathscr{F}(\mathcal{M}):&=\Gamma(X,i_l^{\#} \mathcal{M}), \\
 \mathscr{J}(\mathcal{M}):&=T(\mu) \uset{U(\mathfrak{g})^{\mu}} \Gamma(X \times X, \mathcal{M}).
\end{split}
\end{equation}
The theorem follows if we can prove that $\mathscr{F}(\mathcal{M}) \cong \mathscr{J}(\mathcal{M})$. First, we extend the functors $\mathscr{F},\mathscr{J}$ to the category $\WCoh(\mathcal{D}_{\lambda,\mu},G)$; this is the category of coherent $\mathcal{D}_{\lambda,\mu}$-modules that are equivariant $\mathcal{O}$-modules together with $G$-equivariant morphisms. The reason to do this is that $\mathcal{D}_{\lambda,\mu} \in \WCoh(\mathcal{D}_{\lambda,\mu},G)$, but $\mathcal{D}_{\lambda,\mu} \notin\Coh(\mathcal{D}_{\lambda,\mu},G)$.

Since $i_l^{\#} \cong i_l^*$ as $\mathcal{O}$-modules, $i_l^{\#}$ is an equivalence of Abelian categories and  $\lambda$ is dominant, we obtain by Theorem \ref{d2equivariantbb} that $\mathscr{F}$ is right exact. Furthermore since $\lambda$ and $\mu$ are dominant weights, by the same theorem applied on $\mathcal{D}_{\lambda,\mu}$-modules, $\mathscr{J}$ is also right exact.

Next, let us construct a map from $\mathscr{J}(\mathcal{M}) \to \mathscr{F}(\mathcal{M})$. Let $\mathcal{N}=i_{l_*}i_l^* \mathcal{D}_{\lambda,\mu}$ and consider the space $\mathcal{N} \uset{ \mathcal{D}_{\lambda,\mu}} \mathcal{M}$. By construction, this is isomorphic to $i_{l_*}i_l^* \mathcal{M}$. Therefore, we have

\begin{equation}
\begin{split}
\mathscr{F}(\mathcal{M})&= \Gamma(X,i_l^{\#} \mathcal{M}) \\
              &= \Gamma(X,i_l^* \mathcal{M}) \\
              &\cong \Gamma(X \times X, i_{l_*} i_l^* \mathcal{M})   \\    
              &\cong \Gamma(X \times X, \mathcal{N}  \uset{\mathcal{D}_{\lambda,\mu}} \mathcal{M}). 
\end{split}
\end{equation}

Therefore, it is enough to construct a map from $\mathscr{J}(\mathcal{M})$ to $\Gamma(X \times X, \mathcal{N} \uset{\mathcal{D}_{\lambda,\mu}} \mathcal{M})$. This reduces to proving that $\mathscr{J}(\mathcal{M}) \cong \Gamma(X \times X,\mathcal{N}) \uset{\Gamma(X \times X,\mathcal{D}_{\lambda,\mu})} \Gamma(X \times X,\mathcal{M})$. Let $\mathcal{B}:=\Gamma(X \times X,\mathcal{N}) \uset{\Gamma(X \times X,\mathcal{D}_{\lambda,\mu})} \Gamma(X \times X,\mathcal{M})$.

We have 

\begin{equation}
\begin{split}
\mathcal{B}&=\Gamma(X \times X,\mathcal{N}) \uset{\Gamma(X \times X,\mathcal{D}_{\lambda,\mu})} \Gamma(X \times X,\mathcal{M})\\ &= \Gamma(X \times X,i_{l_*} i_l^* \mathcal{D}_{\lambda,\mu}) \uset{\Gamma(X \times X,\mathcal{D}_{\lambda,\mu})} \Gamma(X \times X,\mathcal{M}) \\
     &\cong U(\mathfrak{g})^{\lambda} \uset{K} T(\mu) \uset{U(\mathfrak{g})^{\lambda} \uset{K} U(\mathfrak{g})^{\mu}} \Gamma(X \times X,\mathcal{M})  \\
      &\cong  T(\mu) \uset{U(\mathfrak{g})^{\mu}} \Gamma(X \times X, \mathcal{M}) \\
      &\cong \mathscr{J}(\mathcal{M}).
\end{split}
\end{equation}

Finally, we have that

\begin{equation}
\begin{split}
\mathscr{J}(\mathcal{D}_{\lambda,\mu}) &= T(\mu) \uset{U(\mathfrak{g})^{\mu}} \Gamma(X \times X, \mathcal{D}_{\lambda,\mu}) \\
          & \cong T(\mu) \uset{U(\mathfrak{g})^{\mu}} U(\mathfrak{g})^{\lambda} \uset{K} U(\mathfrak{g})^{\mu} \\
       & \cong U(\mathfrak{g})^{\lambda} \uset{K} T(\mu) \\
          & \cong \Gamma(X,i_l^{\#} \mathcal{D}_{\lambda,\mu}) \text{ (by equation \eqref{d2globalsectionDlambdalambdamtdo})} \\
          & \cong \mathscr{F}(\mathcal{D}_{\lambda,\mu}).
\end{split}
\end{equation}

The claim follows by picking a presentation $(\mathcal{D}_{\lambda,\mu})^n \to (\mathcal{D}_{\lambda,\mu})^m \to \mathcal{M} \to 0$ and applying the Five Lemma.
\end{proof}

As a corollary, we obtain using Corollary \ref{d2twistedpullbacksimplified}:

\begin{corollary}
Let $\lambda,\mu$ be dominant weights and let $\mathcal{M} \in \Coh(\mathcal{D}_{\lambda,\mu},G)$. Then:

$$\Gamma(X \times X,p_r^{-1} (\mathcal{M}_{\mu}) \uset{p_r^{-1} \mathcal{D}_{\mu}} \mathcal{M}) \cong  T(\mu) \uset{U(\mathfrak{g})^{\mu}} \Gamma(X \times X, \mathcal{M}).$$

\end{corollary}


\section{Primitive/prime ideals in the universal enveloping algebra of a semisimple Lie algebra}
\label{d2sectionproofofDuflos}

We retain the notations from \ref{d2sectionglosec}, recall that $R=K$ is a field of characteristic $0$. In this section, we aim to finish the proof of Duflo's Theorem, that is we aim to prove:

\begin{theorem}
\label{d2twistedduflotheorem}
Let $\lambda: \mathfrak{h} \to K$ be a dominant weight and let $I$ be a primitive/prime ideal in $U(\mathfrak{g})^{\lambda}$. Then

           $$I=\Ann(L(\mu)) \text{ for some $\mu \in \mathfrak{h}^*$}.$$
\end{theorem}

For the rest of this section, fix $\lambda \in \mathfrak{h}^*$ a dominant weight and let $\lambda^*=-w_{o} \lambda$; if $\lambda$ is dominant, $\lambda^*$ will also be dominant. Recall further that $\tau$ induces an isomorphism $ U(\mathfrak{g})^{\lambda^{\op}} \cong U(\mathfrak{g})^{\lambda^*}  $.

Consider the functor $\mathscr{F}:\Mod_{\fg}(U(\mathfrak{g} \times \mathfrak{g})^{\lambda^*,\lambda},G) \to \Mod_{\fg}(U(\mathfrak{g})^{\lambda^*},B)$ defined by

$$\mathscr{F}(M):= \Gamma(X,i_l^{\#} \circ \Loc^{\lambda^*,\lambda}(M)).$$

\begin{proposition}
\label{d2functorFisexact}
The functor $\mathscr{F}$ is exact and $\mathscr{F}(M) \cong T(\lambda) \uset{U(\mathfrak{g})^{\lambda}} M$ as $U(\mathfrak{g})^{\lambda^*}$-modules.

\end{proposition}

\begin{proof}

Let $\mathcal{M}:=\Loc^{\lambda^*,\lambda}(M)$; we have by Corollary \ref{d2glolocisomor} that $\Gamma(X \times X, \mathcal{M}) \cong M$ and by Theorem \ref{d2equivariantbb} that $\mathcal{M} \in \Coh(\mathcal{D}_{\lambda^*,\lambda},G)$ , so the second claim follows from Theorem \ref{d2globalsectionspullback}. Consider a short exact sequence of finitely generated $G$-equivariant $U(\mathfrak{g} \times \mathfrak{g})^{\lambda^*,\lambda}$-modules:

$$ 0 \to N \to M \to P \to 0.$$

We let $\mathcal{N},\mathcal{M}, \mathcal{P}$ be the localisations of $N,M$ and $P$ respectively and let $\mathcal{K}=\ker(\mathcal{N} \to \mathcal{M})$. Since the functor $\Loc^{\lambda^*,\lambda}$ is right exact we obtain an exact sequence

$$ 0 \to \mathcal{K} \to \mathcal{N} \to \mathcal{M} \to \mathcal{P} \to 0.$$

Since $i_l^{\#}$ is an equivalence of Abelian categories by Theorem \ref{d2alaBorho-Brylinski}, $i_l^{\#}$ is in particular exact. Furthermore, by Theorem \ref{d2equivariantbb} the global sections functor is also exact, so we obtain an exact sequence

$$ 0 \to \Gamma(X,i_l^{\#} \mathcal{K}) \to \Gamma(X,i_l^{\#} \mathcal{N}) \to \Gamma(X,i_l^{\#} \mathcal{M}) \to \Gamma(X,i_l^{\#} \mathcal{P}) \to 0.$$

Applying Theorem \ref{d2globalsectionspullback} we obtain an exact sequence
\begin{equation}
\begin{split}
0 &\to T(\lambda) \uset{U(\mathfrak{g})^{\lambda}} \Gamma(X \times X,\mathcal{K}) \to T(\lambda) \uset{U(\mathfrak{g})^{\lambda}} N \to \\
&\to T(\lambda) \uset{U(\mathfrak{g})^{\lambda}} M \to T(\lambda) \uset{U(\mathfrak{g})^{\lambda}} P \to 0.
\end{split}
\end{equation}

The claim follows since $\Gamma(X \times X, \mathcal{K})=0$ by definition of  $\mathcal{K}$, Theorem \ref{d2equivariantbb} and Corollary \ref{d2glolocisomor}.
\end{proof}

\begin{lemma}
\label{d2Fisalmostinjectiveonobjects}

Let $M \in \Mod_{\fg}(U(\mathfrak{g} \times \mathfrak{g})^{\lambda^*,\lambda},G)$ and assume that $\mathscr{F}(M)=0$. Then $M=0$.

\end{lemma}

\begin{proof}

Let $\mathcal{M}=\Loc^{\lambda^*,\lambda}(M)$. Then, by assumption, we have that $\Gamma(X,i_l^{\#} \mathcal{M})=0$. Applying Corollary \ref{d2zeroglobalsectionsil}, we obtain $\Gamma(X \times X, \mathcal{M})=0$, so by Corollary \ref{d2glolocisomor}, $M=0$.
\end{proof}

We now specialise to two sided ideals in $U(\mathfrak{g})^{\lambda}$; recall that a two-sided ideal $I$ can be viewed as a module over $U(\mathfrak{g})^{\lambda^*} \otimes U(\mathfrak{g})^{\lambda}$ via $(x \otimes y).i=yi\tau(x)$ for $x \in U(\mathfrak{g})^{\lambda^*},y \in U(\mathfrak{g})^{\lambda}$ and $i \in I$. Further, by Corollary \ref{d2lambdaequivariancetwosided}, $I \in \Mod_{\fg}(U(\mathfrak{g} \times \mathfrak{g})^{\lambda^*,\lambda},G)$, so $\mathscr{F}(I)$ is well-defined. As a corollary, we obtain immediately:

\begin{corollary}
\label{d2Fpreservesstrictinclusionofideals}

Let $I,J$ be two-sided ideals in $U(\mathfrak{g})^{\lambda}$ such that $I \subseteq J$. Assume that $\mathscr{F}(I) \cong \mathscr{F}(J)$. Then $I=J$.

\end{corollary}

\begin{proof}

Consider the short exact sequence:

$$0 \to I \to J \to J/I \to 0.$$

By Proposition \ref{d2functorFisexact}, $\mathscr{F}$ is an exact functor, so we obtain an exact sequence

$$0 \to \mathscr{F}(I) \to \mathscr{F}(J) \to \mathscr{F}(J/I) \to 0.$$ 

By assumption, we have $\mathscr{F}(I) \cong \mathscr{F}(J)$, so $\mathscr{F}(J/I)=0$. The claim follows by Lemma \ref{d2Fisalmostinjectiveonobjects}.
\end{proof}

\begin{lemma}
\label{d2computeFonideals}
Let $I$ a two-sided ideal in $U(\mathfrak{g})^{\lambda}$. Then as $U(\mathfrak{g})^{\lambda^*}$-modules we have:
  
            $$\mathscr{F}(I) \cong  T(\lambda)I.$$

\end{lemma}

We should remark that $U(\mathfrak{g})^{\lambda^*}$ acts on $T(\lambda)I$ via $x.(ti)=t(x.i)=ti\tau(x)$ for $x \in U(\mathfrak{g})^{\lambda^*}, t \in T(\lambda), i \in I$ and the isomorphism is natural in $I$.
\begin{proof}
 Consider the following short exact sequence:

$$ 0 \to I \to U(\mathfrak{g})^{\lambda} \to U(\mathfrak{g})^{\lambda}/I \to 0.$$

Applying the exact functor $\mathscr{F}$ and using Proposition \ref{d2functorFisexact}, we obtain an exact sequence:

$$0 \to  T(\lambda) \uset{U(\mathfrak{g})^{\lambda}} I \to T(\lambda) \to T( \lambda) \uset{U(\mathfrak{g})^{\lambda}} U(\mathfrak{g})^{\lambda}/I \to 0.  $$

This exact sequence fits into the commutative diagram:

\begin{center}
\begin{tikzcd}
&0 \arrow[r] &T(\lambda) \uset{U(\mathfrak{g})^{\lambda}} I \arrow[d] \arrow[r] & T(\lambda)   \arrow[r] \arrow[d]  &T(\lambda) \uset{U(\mathfrak{g})^{\lambda}} U(\mathfrak{g})^{\lambda}/I \arrow[r] \arrow[d] &0 \\
&0 \arrow[r] & T(\lambda)I   \arrow[r] & T(\lambda)  \arrow[r] & T( \lambda) \uset{U(\mathfrak{g})^{\lambda}} U(\mathfrak{g})^{\lambda}/I \arrow[r] &0.
\end{tikzcd}
\end{center}

It is easy to see that the first map is a surjection and the second and the third maps are isomorphisms. Furthermore, by the Five Lemma, the first map is injective, so indeed we get $\mathscr{F}(I) \cong T(\lambda)I $. 
\end{proof}

 Let $\sigma: \mathfrak{g} \to \mathfrak{g}$ denote the Chevalley involution that swaps $\mathfrak{n}^+$ and $\mathfrak{n}^{-}$ and fixes $\mathfrak{h}$ . Then $\sigma$ extends to an anti-automorphism of $U(\mathfrak{g})$ that fixes the center by \cite[Exercise 1.10]{Hu1}, so it descends to an anti-automorphism $\sigma:U(\mathfrak{g})^{\lambda} \to U(\mathfrak{g})^{\lambda}$. Recall that by construction we have $T(\lambda)=K_{\lambda} \uset{U(\mathfrak{b}^{-})} U(\mathfrak{g})$.
 
\begin{lemma}
\label{d2TlambdaMlambdaiso}

The map

$$\phi:T(\lambda) \to M(\lambda), \qquad \phi(k \otimes x)=\sigma(x) \otimes k, \quad k \in K_{\lambda}, x \in U(\mathfrak{g})$$

is a $K$-linear isomorphism of vector spaces satisfying $\phi(tu)=\sigma(u)\phi(t)$ for all $u \in U(\mathfrak{g})$ and $t \in T(\lambda)$.

\end{lemma}

\begin{proof}

Since $\sigma$ swaps $\mathfrak{n}^+$ and $\mathfrak{n}^{-}$ and fixes $\mathfrak{h}$, it is easy to check that the map $\phi$ is well-defined. Furthermore, as $\sigma$ is a $K$-linear anti-automorphism, $\phi$ is a $K$-linear isomorphism of vector spaces.

Finally, we have for $k \otimes x \in T(\lambda)$ and $u \in U(\mathfrak{g})$:

\begin{equation}
\begin{split}
\phi((k \otimes x)u)&=\phi(k \otimes xu)\\
                    &=\sigma(xu) \otimes k\\
                    &=\sigma(u)\sigma(x) \otimes k\\
                    &=\sigma(u)\phi(k \otimes x). \qedhere
\end{split}
\end{equation}

\end{proof} 

In particular, we obtain that if $I$ is a two-sided ideal in $U(\mathfrak{g})^{\lambda}$, $\phi(T(\lambda)I)=\sigma(I)M(\lambda).$

\begin{corollary}
\label{d2annihilatorofidealverma}
Let $I$ a two-sided ideal in $U(\mathfrak{g})^{\lambda}$. Then 
              $$I=\Ann(M(\lambda)/IM(\lambda)).$$

\end{corollary}

\begin{proof}

Let $J:=\Ann(M(\lambda)/IM(\lambda))$. Since $ I(M(\lambda)/IM(\lambda))=0,$ we obtain that $I \subseteq J$, so $\sigma(I) \subseteq \sigma(J)$. We remark that since $\sigma$ is anti-automorphism of $U(\mathfrak{g})^{\lambda}$, $\sigma(I)$ and $\sigma(J)$ are two-sided ideals in $U(\mathfrak{g})^{\lambda}$. Since $\mathscr{F}$ preserves injections, we obtain $\mathscr{F}(\sigma(I)) \subseteq \mathscr{F}(\sigma(J))$.

Consider the following  diagram:

\begin{center}
\begin{tikzcd}
&\mathscr{F}(\sigma(I)) \arrow[r] \arrow[d] & \mathscr{F}(\sigma(J)) \arrow[d]\\
&T(\lambda)\sigma(I) \arrow[r] \arrow[d,"\phi"] & T(\lambda) \sigma(J) \arrow[d,"\phi"]\\
&IM(\lambda) \arrow[r] &JM(\lambda)

\end{tikzcd}
\end{center}

By construction, the bottom diagram commutes and by the definition of $J$, we have $JM(\lambda)=IM(\lambda)$. Using Lemma \ref{d2TlambdaMlambdaiso}, we get $T(\lambda)\sigma(I)=T(\lambda)\sigma(J)$. Furthermore, since the isomorphism in Lemma \ref{d2computeFonideals} is natural on ideals, we obtain that the top diagram also commutes. Therefore, $\mathscr{F}(\sigma(I)) \cong \mathscr{F}(\sigma(J))$. The claim now follows by Corollary \ref{d2Fpreservesstrictinclusionofideals} since $\sigma$ is an anti-automorphism.
\end{proof}

We now have all the ingredients to prove Duflo's Theorem; the idea of the proof was first given by Dixmier in \cite[Problem 30]{Dix}.
\begin{proof}[Proof of Theorem \ref{d2twistedduflotheorem}]

Let $I$ be a prime ideal in $U(\mathfrak{g})^{\lambda}$. Then, by Corollary \ref{d2annihilatorofidealverma}, we have $I=\Ann(M(\lambda)/IM(\lambda))$. Since $M(\lambda)/IM(\lambda)$ is quotient of $M(\lambda)$, we have by Proposition \ref{d2vermammodulesproperties} that there exists a finite composition series 

 $$0=M_0 \subset M_1 \subset M_2 \subset \ldots \subset M_n=M(\lambda)/IM(\lambda).$$

Let $I_i:=\Ann(M_i/M_{i-1})$ for $1 \leq i \leq n$. Then

\begin{equation}
\begin{split}
   I_1I_2 \ldots I_n(M(\lambda)/IM(\lambda))&= I_1I_2 \ldots I_n M_n \\
                                            &=I_1I_2 \ldots I_{n-1} M_{n-1} \\
                                            &= \ldots \\
                                            &=0,
\end{split}
\end{equation}
so $I_1I_2 \ldots I_n \subset I$. Since the ideal $I$ is prime, there exists $1 \leq j \leq n$ such that $I_j \subset I$. On the other hand, by construction $I \subset I_j$. Thus, we obtain that $I=I_j=\Ann(M_j/M_{j-1})$. The claim follows since by Proposition \ref{d2vermammodulesproperties}, $M_j/M_{j-1} \cong L(\mu)$ for some $\mu \in\mathfrak{h}^*.$

Since any primitive ideal is in particular prime, we obtain that any primitive ideal in $U(\mathfrak{g})^{\lambda}$ is the annihilator of some $L(\mu)$.
\end{proof}

In the case when $K$ is algebraically closed, for example, $K=\mathbb{C}$, we may characterise all the primitive ideals in $U(\mathfrak{g})$.

\begin{corollary}
Assume that $K$ is an algebraically closed field of characteristic $0$ and let $I$ be a primitive ideal in $U(\mathfrak{g})$. Then 
 
            $$I=\Ann(L(\mu)) \text{ for some } \mu:\mathfrak{h} \to K.$$

\end{corollary} 

\begin{proof}
This follows by combining Proposition \ref{d2vermammodulesproperties}, Theorem \ref{d2twistedduflotheorem} and Theorems \cite[8.4.3-8.4.4 d)]{Dix}.
\end{proof}

\bibliography{duflo2}
\bibliographystyle{plain}

\end{document}